\documentclass[12pt]{amsart}
\usepackage[margin=1.0in]{geometry}
\usepackage{amsmath,amssymb,amsfonts,graphicx,color,fancyhdr,psfrag,comment,enumerate,mathabx}
\usepackage[latin1]{inputenc}
\usepackage[hyperpageref]{backref}
\usepackage[colorlinks=true, pdfstartview=FitV, linkcolor=blue, citecolor=blue, urlcolor=blue]{hyperref}
\usepackage[capitalise,noabbrev]{cleveref}
\usepackage{tcolorbox,longfbox}
\allowdisplaybreaks
\usepackage{epstopdf}
\usepackage{mathrsfs}
\usepackage{epsfig}
\usepackage{hyperref}

\usepackage{ifthen}
\usepackage{float}
\usepackage{enumerate}
\usepackage{mathtools}
\usepackage[bottom]{footmisc}
\usepackage{tikz}
\usepackage{comment}
\usetikzlibrary{matrix,shapes,arrows,positioning,chains}
\usepackage{amssymb,amsmath,amsfonts}
\usepackage{bbm}

\numberwithin{equation}{section}

\newtheorem{theorem}{Theorem}[section]
\newtheorem{proposition}[theorem]{Proposition}
\newtheorem{lemma}[theorem]{Lemma}

\newcommand{\R}{\mathbb{R}}
\newcommand{\C}{\mathbb{C}}
\newcommand{\h}{\mathbb{H}}

\begin{document}
\title[Grushin operator with drift]{Riesz transforms associated with the \\ Grushin operator with drift}

\author[N. Garg and R. Garg]
{Nishta Garg \and Rahul Garg}

\address[N. Garg]{Department of Mathematics, Indian Institute of Science Education and Research Bhopal, Bhopal--462066, Madhya Pradesh, India.}
\email{nishta21@iiserb.ac.in}

\address[R. Garg]{Department of Mathematics, Indian Institute of Science Education and Research Bhopal, Bhopal--462066, Madhya Pradesh, India.}
\email{rahulgarg@iiserb.ac.in}

\subjclass[2020]{Primary: 42B20. Secondary: 22E25, 58J35}
\keywords{Riesz transforms, Heat semigroup, Grushin operator, Heisenberg-Reiter group.}

\begin{abstract}
We consider the Grushin operator with drift which is symmetric with respect to a measure having exponential growth.  For the corresponding Riesz transforms, we study strong-type $(p, p)$, $1 < p < \infty$, and weak-type $(1, 1)$ boundedness. 
\end{abstract}

\maketitle

\section{Introduction}
In the present work, we are concerned with the Grushin operator on $\R^{n+m}$. These subelliptic operators were introduced in 1970 in \cite{Grushin-Math-Sb-1970} on $\R^2$. The Grushin operator $G$ induces a sub-Riemannian metric $\tilde{d}$ on $\R^{n+m}$ and it is well known from the work of \cite{Analysis-degenerate-elliptic-opertators-Robinson-Sikora-Math-Z-2008} that the space $(\R^{n+m}, \, \tilde{d}, \, dx)$ is of homogeneous type in the sense of Coifman--Weiss \cite{Coifman-Weiss-book-1971}, that is, it is a doubling metric measure space. 
Over the last decade or so, there has been a lot of work done on various problems concerning the Grushin operator. 
In particular, their spectral multipliers have been systematically studied in a series of works \cite{Riesz-trans-multipliers-Grushin-oper-Jotsaroop-Sanjay-thangavelu-JDM-2014, Martini-Sikora-sharp-multiplier-Grushin-MRL-2012, Martini-Muller-sharp-multiplier-Grushin-Revista-2014, Dallara-Martini-robust-approach-multiplier-grushin-TAMS-2020, Dallara-Martini-aptimal-multiplier-grushin-part-II-JFAA-2022, Dallara-Martini-aptimal-multiplier-grushin-part-I-Revista-2023}, which have also motivated some closely related works on the pseudo-multipliers \cite{Bagchi-Basak-Garg-Ghosh-sparse-pseudo-multiplier-grushin-I-JFAA-2023, Bagchi-Basak-Garg-Ghosh-sparse-pseudo-multiplier-grushin-II-JGA-2024, Bagchi-Basak-Garg-Ghosh-sparse-pseudo-multiplier-grushin-III-JMAA-2024}. 

\medskip 
Several authors have studied Riesz transforms associated with the Grushin operator $G$. Let us recall some of the important works done in this context. 
In the limited case of $m=1$, the boundedness of these operators was first established by Jotsaroop et al \cite{Riesz-trans-multipliers-Grushin-oper-Jotsaroop-Sanjay-thangavelu-JDM-2014} on $L^p(\R^{n+1})$, for $1< p< \infty$. Shortly afterwards, Sanjay--Thangavelu \cite{Dim-free-Riesz-trans-Grushin-oper-PAMS-Sanjay-thangavelu-2014} showed their dimension-free boundedness. Later on, Robinson--Sikora \cite{Riesz_transform_Robinson} proved the $L^p$-boundedness for $1<p<\infty$ and weak-type $(1,1)$ of these Riesz transforms in a more general setup. The key idea in their work is to relate these Riesz transforms with those for suitable nilpotent Lie groups and with that they could build upon the arguments adapting Coifman--Weiss transference techniques \cite{Transference_methods_in_analysis}. This approach has also been successfully utilised in the works \cite{Martini-Sikora-sharp-multiplier-Grushin-MRL-2012, Dziubanski-Jotsaroop-H1-BMO-Grushin-JFAA-2016}. We would also like to refer to \cite{Dziubanski-Sikora-Lie-group-approach-JLT-2021} where this beautiful technique has been explored further to show its usefulness in studying a broad class of examples. 

\medskip 
The aim of this article is to consider the Grushin operator with drift, which is symmetric with respect to a measure having exponential growth, and study strong-type $(p,p)$ and weak-type $(1,1)$ boundedness properties of the associated Riesz transforms of first and higher orders. Before describing our setup and the results, let us first do a quick survey of the relevant literature, particularly for the Laplacian with drift on the Euclidean space. 

\medskip 
Given a non-zero vector $\nu = ( \nu_1, \ldots, \nu_n) \in \R^n$, consider the Laplacian with drift given by 
$$
\Delta_{\nu} = - \Delta - 2 \, \nu \cdot \nabla = - \sum_{j=1}^{n} \frac{\partial^2}{\partial x_j^2} - 2 \sum_{j=1}^n \nu_j \frac{\partial}{\partial x_j}.
$$
It is well known that $\Delta_{\nu}$ is positive-definite and essentially self-adjoint on $L^2(\R^n, \, d\mu_\nu)$, where $d\mu_{\nu}(x) = e^{2\nu\cdot x} \, dx$. 
Clearly, the measure $d\mu_{\nu}$ is of exponential volume growth, and therefore the classical Calder\'{o}n-Zygmund theory is not directly applicable to study the corresponding Riesz transforms $R_\alpha  = \nabla^\alpha (\Delta_\nu)^{-|\alpha|/2}$. 

\medskip 
Such Riesz transforms were first studied by Lohou\'{e}--Mustapha \cite{Lohoue-Mustapha-drift-2004}. In fact, the work of Lohou\'{e}--Mustapha is in a more general context, namely on the amenable groups. Confining ourselves to the Euclidean space set-up, the following is their result. 
\begin{theorem}[Lohou\'{e}--Mustapha \cite{Lohoue-Mustapha-drift-2004}] 
\label{thm:Lohoue-Mustapha-2004-result}
For any $1< p <\infty$ and $\alpha \in \mathbb{N}^n$ with $|\alpha| \geq 1$, the Riesz transform $R_\alpha  = \nabla^\alpha (\Delta_\nu)^{-|\alpha|/2}$ is bounded on $L^p(\R^n, d\mu_\nu).$
\end{theorem} 

The endpoint case, that is, the boundedness properties for $p=1$, was first studied by  Li--Sj\"ogren--Wu \cite{Li-Sjogren-Wu-drift-Euclidean-Math-Z-2016} and was further strengthened by Li--Sj\"ogren \cite{Li-Sjogren-drift-sharp-endpoint-Euclidean-Canad-2021}. Let $R_D = D(\Delta_\nu)^{-k/2}$ be a $k^{th}$-order Riesz transform where $D$ is a homogeneous differential operator of degree $k \geq 1$ with constant coefficients. One can rewrite $D$ in terms of the derivatives in the directions of $\nu$ and its orthogonal complement. With that, Li--Sj\"ogren \cite[Theorem 1.1]{Li-Sjogren-drift-sharp-endpoint-Euclidean-Canad-2021} proved that $R_D$ is of weak-type $(1,1)$ with respect to $d\mu_\nu$ if and only if the maximum number of derivatives in the direction of $\nu$ are at most 2. Before moving on, let us mention that in recent years there have been some developments on related problems in other contexts. See, for example, \cite{Li-Sjogren-drift-real-hyperbolic-Potential-Anal-2017, Heisenberg_group_drift, Betancor-Laplacian-drift-Math-Nachr-2025, Garg-Riesz-transform-twisted-laplacian-JFAA-2025}. 

\medskip 
Coming back to the Grushin operator $G$, let us recall that $G$ is a second-order hypoelliptic partial differential operator defined by 
$$
G = -\sum_{j=1}^n X_j^2 - \sum_{j=1}^n \sum_{k=1}^m X_{j,k}^2,
$$
where, by writing points in $\R^{n+m}$ as $x = (x',x'')$ with $x'=(x_1',x_2',\ldots,x_n') \in \R^n$ and $x''=(x_1'',x_2'',\ldots,x_m'') \in \R^m$, the first order vector fields are given by 
\begin{equation} 
\label{def-first-order-grushin-gradients}
X_j = \frac{\partial}{\partial x'_j} \qquad \text{and} \qquad X_{j,k}= x'_j \, \frac{\partial}{\partial x''_k}, 
\end{equation} 
where $1 \leq j \leq n$ and $1 \leq k \leq m.$ 

\medskip 
Let us write $ X = (X' ,\, X'') = (X_1, \ldots, X_n, \,  X_{1,1}, \ldots, X_{n,m}).$ Then, given a non-zero vector $\nu = (a,b) \in \R^{n} \times \R^{nm}$, we define the Grushin operator with drift $G_{\nu}$ by 
\begin{align}
\label{def:Grushin_with_drift}
G_\nu = G - 2 \, a \cdot X' - 2 \, b \cdot X''.
\end{align} 

It turns out (see Proposition \ref{prop:auto-def-grushin-drifted}) that the Grushin operator with drift $G_\nu$ is symmetric with respect to a positive measure $\mu$ on $\R^{n+m}$ if and only if $b=0$ and upto a scalar multiple, the measure $\mu$ is given by $d\mu = d\mu_a = e^{2a\cdot x'} \, dx$, where $dx$ denote the Lebesgue measure on $\R^{n+m}$. Henceforth, we shall denote the Grushin operator with drift as 
$$ G_a = G - 2 \, a \cdot X' = G - 2 \, a \cdot \nabla_{x'} $$ 
with $a$ a non-zero vector in $\R^n$. It also turns out that the operator $G_a$ is essentially self-adjoint with respect to the measure $d\mu_a = \, e^{2a\cdot x'} \, dx$. 

\medskip 
In this article, we study the Riesz transforms of arbitrary order $k \geq 1$ associated with $G_a$ which we now define. Given a multi-index $\alpha \in \mathbb{N}^{n+nm}$ such that $|\alpha| = k \geq 1$, let us define the Riesz transform $R_{\alpha, a}$ of order $k$ by 
$$ R_{\alpha, a} = X^\alpha G_a^{-k/2}.$$ 

We have the following result concerning the $L^p$-boundedness of these Riesz transforms. 
\begin{theorem} \label{thm:p_bdd}
For any multi-index $\alpha$, such that $|\alpha|=k \geq 1$ and $1<p<\infty$, the Riesz transform $R_{\alpha,a}$ is bounded on $L^{p}(\R^{n+m}, d\mu_a),$ uniformly in $a$.   
\end{theorem}

And, for the endpoint case $p=1$, we have the following two results.
\begin{theorem} \label{thm:1_bdd-positive}
On $\R^{n+1}$, that is, for $m=1$, the first order Riesz transforms are of weak-type $(1,1)$ with respect to $d\mu_a$, uniformly in $a$.
\end{theorem}

\begin{theorem} \label{thm:1_bdd-negative}
For any $k \geq 3$, not all the Riesz transforms of order $k$ are of weak-type $(1,1)$ with respect to $d\mu_a$. 
\end{theorem}

\medskip 
\textbf{Organization of the paper:} 
In Section \ref{sec:prelim}, we start with recalling relevant basics on the Grushin operator, the sub-Laplacian on some Heisenberg-Reiter groups, and some facts on how they are related. After that, we discuss the Grushin operator with drift, the Riesz transforms associated with it, and also the asymptotics of the corresponding ball volume in Subsection \ref{subsec:Grushin-operator-with-drift}. 
In Subsection \ref{subsec:rotation-dilation}, we shall explain that by making use of dilation and rotation arguments, it suffices to study the problems for the unit vector $e_1$. 

\medskip 
We shall prove Theorems \ref{thm:p_bdd} and \ref{thm:1_bdd-positive} in Section \ref{sec:Reg-Riesz-transforms-transference-principle}. The main idea of proofs of these results is based on the Coifmann-Weiss transference technique \cite{Transference_methods_in_analysis} and follows closely the approach developed in \cite{Riesz_transform_Robinson}. Doing so, we shall deduce our results for the Grushin operator with drift from the known results (Theorems \ref{thm:Lohoue-Mustapha-Riesz-Lp} and \ref{thm:LS-Riesz-Heisenberg}) for the sub-Laplacian with drift on the Heisenberg-Reiter groups. 

\medskip 
We shall show in Proposition 
\ref{prop:pointwise-convergence} that a Riesz transform for Grushin operator with (scaled) drift, when conjugated by suitable Euclidean translation and dilations, converges pointwise to the Riesz transform for the Laplacian with drift on $\R^{n+m}$. With this one can easily show that the boundedness of the Riesz transform for Grushin operator with drift implies the boundedness of the Riesz transform for the Laplacian with drift on $\R^{n+m}$, and then Theorem \ref{thm:1_bdd-negative} holds true via Theorem 1.1 of \cite{Li-Sjogren-drift-sharp-endpoint-Euclidean-Canad-2021}. 


\section{Preliminaries and Basic results} 
\label{sec:prelim}

\subsection{Grushin operator} \label{prelim:grushin-op}
Let us denote arbitrary $x \in \R^{n+m}$ by $x = (x',x'')$, where $x'=(x_1',x_2',\ldots,x_n') \in \R^n$ and $x''=(x_1'',x_2'',\ldots,x_m'') \in \R^m$. With this, the Grushin operator $G$ on $\R^{n+m}$ is given by 
\begin{equation}
G = -\Delta_{x'} - |x'|^2 \Delta_{x''} = -\sum_{j=1}^n \frac{\partial^2}{\partial x^{'2}_j} - |x'|^2 \sum_{k=1}^m \frac{\partial^2}{\partial x^{''2}_k}.
\end{equation}
Clearly, the operator $G$ is the negative of the sum of squares of the following vector fields: 
$$
X_j = \frac{\partial}{\partial x'_j} \qquad \text{and} \qquad X_{j,k}= x'_j \, \frac{\partial}{\partial x''_k}, \qquad 1 \leq j \leq n, \, \, 1 \leq k\leq m.
$$
For future purposes, let us write 
\begin{equation} 
\label{def:grushin-vector-fields}
X = (X', \, X'') = (X_1,\ldots, X_n, \, X_{1,1}, \ldots, X_{n,m}). 
\end{equation}

One can easily check that the hypoelliptic operator $G$ is homogeneous of degree 2 with respect to the family of non-isotropic dilations $\delta_r(x',x'') = (r x',r^2 x''), \, r>0$. 

\medskip The control distance $\tilde{d}(x,y)$ of the Grushin operator has the following asymptotic behavior (see \cite[Proposition 5.1]{Analysis-degenerate-elliptic-opertators-Robinson-Sikora-Math-Z-2008}): 
\begin{equation} 
\label{Grushin_distance}
\tilde{d}(x,y) \sim  d(x,y) := |x'-y'| + \left\{
\begin{array}{ll}
\frac{|x''-y''|}{|x'|+|y'|}, & \mbox{if } |x''-y''|^{1/2} \leq |x'|+|y'| \\
|x''-y''|^{1/2},& \mbox{if } |x''-y''|^{1/2} \geq |x'|+|y'|.
\end{array} \right.    
\end{equation}
From now onwards, we shall call $d(x,y)$ the Grushin metric. Let $B(x,r)$ denote the ball with center $x$ and radius $r$, that is, $B(x,r) = \, \{y \in \R^{n+m}: d(x,y)<r\}$. The Lebesgue measure of $B(x,r)$ is known to have the following asymptotics (again, see \cite[Proposition 5.1]{Analysis-degenerate-elliptic-opertators-Robinson-Sikora-Math-Z-2008}): 
\begin{equation} 
\label{eq:ball_volume}
|B(x,r)| \sim r^{n+m} \, \max \{r, \, |x'|\}^m \sim r^{n+m}(r+|x'|)^m. 
\end{equation}
As an immediate consequence of \eqref{eq:ball_volume}, we get that $( \R^{n+m}, \, \tilde{d}, \, dx )$ is a doubling metric measure space with homogeneous dimension $Q:=n+2m$. 

\medskip 
Grushin operator is intimately related to the scaled Hermite operators $H(\lambda) = -\Delta_{x'}+|\lambda|^2|x'|^2$ on $\R^n$ in the following manner. For Schwartz class functions $f$ on $\R^{n+m},$ one can apply the Euclidean Fourier inversion formula in $x''$-variable to see that 
\begin{equation} 
\label{def:Grushin_Hermite}
Gf(x)= (2\pi)^{-m} \int_{\R^m} e^{-i\lambda\cdot x''} H(\lambda)f^{\lambda}(x') \, d\lambda, 
\end{equation}
where $\displaystyle f^\lambda(x') = \int_{\R^m} f(x',x'') e^{i\lambda\cdot x''}\, dx''$. 

\medskip 
The operator $G$ generates a symmetric diffusion semigroup (heat semigroup) $(e^{-tG})_{t > 0}$ on $L^2(\R^{n+m}),$ and using relation \eqref{def:Grushin_Hermite}, one can express its heat kernel in terms of the heat kernel $k_{t,\lambda}(x,y)$ of the operator $H(\lambda)$ on $\R^n$. More precisely, 
\begin{equation} \label{eq:heat_kernel_in_hermite_form}
H_t(x,y) = (2\pi)^{-m} \int_{\R^m} k_{t,\lambda}(x',y') \, e^{-i\lambda\cdot (x''-y'')} \, d\lambda, 
\end{equation}
with the kernel $k_{t,\lambda}(x',y')$ given by 
\begin{equation} \label{eq:heat_kernel_hermite}
\left( \frac{|\lambda|}{2\pi \sinh{(2|\lambda|t})}\right)^{n/2} \exp{ \left( -\frac{|\lambda|}{2} \coth{(2|\lambda|t)} \, |x'-y'|^2 - |\lambda| \tanh (|\lambda|t) \, x'\cdot y' \right) }.  
\end{equation}
 
It is well known that the heat kernel of the Grushin operator satisfies the Gaussian upper bounds, that is, there exist constants $b, C>0$ such that 
$$
|H_t(x,y)| \leq C |B(x,\sqrt{t})|^{-1} \exp{(-b \, d(x,y)^2/t)}.
$$
For more details on the above estimate, we refer to Corollary 6.6 of \cite{Analysis-degenerate-elliptic-opertators-Robinson-Sikora-Math-Z-2008}. 


\subsection{The Heisenberg-Reiter group} \label{sbsec:Heisenberg_Reiter_Group}
The Grushin operator has a close relation with the sub-Laplacian of certain Heisenberg-Reiter group as we shall see in the next subsection. 

\medskip 
We briefly recall some of facts concerning the specific Heisenberg-Reiter groups that are relevant in the study of the Grushin operators, and for details, we refer to \cite{Martini-Sikora-sharp-multiplier-Grushin-MRL-2012, Heisenberg_Reiter_Torres} and the references therein. 

\medskip 
Let $\R^{n\times m}$ denote the set of all $n \times m$ matrices with real entries. Let $\h_{n,m}$ be the semi-direct product $\R^{n\times m} \rtimes (\R^n \times \R^m)$, with the group law
$$
(u,v,s) \cdot (u',v',s') = (u+ u',v+ v',s+ s'+ (u'\,^T v - u^T v')/2).
$$
The group $\h_{n,m}$ is an example of Heisenberg-Reiter groups. 

\medskip 
Let $\{\tilde{X}_{1,1}, \ldots,\tilde{X}_{n,m}, \, \tilde{Y_1},\ldots,\tilde{Y_n}, \, \tilde{T_1},\ldots,\tilde{T}_{m}\}$ be the standard basis of the Lie algebra of $\h_{n,m}$. With this set of left-invariant vector fields on $\h_{n,m}$, the homogeneous sub-Laplacian $\mathcal{\tilde{L}}$ on $\h_{n,m}$ is given by 
\begin{equation} \label{eq:sub-Laplacian}
\mathcal{\tilde{L}} = -\sum_{j=1}^n\sum_{k=1}^m \tilde{X}_{j,k}^2 -\sum_{j=1}^n \tilde{Y}_{j}^2.
\end{equation}
The system 
$$ \tilde{X} = (\tilde{X}_{1,1}, \ldots,\tilde{X}_{n,m}, \, \tilde{Y_1},\ldots,\tilde{Y_n}) $$ 
satisfies H\"ormander's condition. Thus, if we denote by $\varrho$ the Carnot-Carath\'eodory distance on $\h_{n,m}$, and write $\varrho(g) = \varrho(g,e)$, where $e$ is identity element of $\h_{n,m}$, then we have that the kernel $p_t$ of the heat semigroup $(e^{-t\mathcal{\tilde{L}}})_{t > 0}$ along with its gradients satisfy the following Gaussian upper bounds (see Theorem IV.4.2 in \cite{Analysis_and_geometry_on_groups}): for any multi-index $\gamma$ and $\kappa > 0$, there exists a constant $C = C_{\gamma, \kappa}> 0$, such that for all $g \in \h_{n,m}$ and $t>0$, 
\begin{equation} 
\label{est:gradient-group-heat-kernel-bounds}
|\tilde{X}^\gamma p_t(g)| \leq C \, t^{-|\gamma|/2} \, V(\sqrt{t})^{-1} \, e^{-\frac{\varrho(g)^2}{(4+\kappa)t}}. 
\end{equation}
Here $V(\sqrt{t})$ denotes the volume of balls of radius $\sqrt{t}$. 

\medskip 
In \cite{Hebisch-Mauceri-Meda-spectral-multipliers-drift-Lie-group-Math-Z-2005}, Hebisch et al studied spectral multipliers for sub-Laplacians with drift on connected Lie groups. In particular, given $\nu = (a,b) \in \R^{n} \times \R^{nm}$, the sub-Laplacian with drift $\mathcal{\tilde{L}_\nu}$ on the group $\h_{n,m}$ is defined as 
\begin{equation} \label{def:sub-Laplacian_with_drift}
\mathcal{\tilde{L}_\nu} = \mathcal{\tilde{L}} - 2 \sum_{j=1}^n\sum_{k=1}^m b_{j,k} \tilde{X}_{j,k} -2 \sum_{j=1}^n a_j\tilde{Y}_{j}.    
\end{equation}
It was shown in \cite{Hebisch-Mauceri-Meda-spectral-multipliers-drift-Lie-group-Math-Z-2005} that the operator $\mathcal{\tilde{L}_\nu}$ is positive definite and essentially self-adjoint on $L^2(\h_{n,m}, \, d\mu_\nu')$, where $d\mu_\nu'(u,v,s)= e^{2(b,a)\cdot (u,v)} \, du \, dv \, ds$. 

\medskip 
As also pointed out by Hebisch et al \cite[p. 906]{Hebisch-Mauceri-Meda-spectral-multipliers-drift-Lie-group-Math-Z-2005}, the operator $\mathcal{\tilde{L}_\nu}$ 
generates a convolution semigroup of probability measures on the group $\h_{n,m}.$ Hence, it follows from Hunt's theorem \cite{Hunt-Semi-groups-measures-Lie-groups-TAMS-1956} that $\mathcal{\tilde{L}_\nu}$ generates a symmetric diffusion semigroup $(e^{-t\mathcal{\tilde{L}_\nu}})_{t>0}$ on $(\h_{n,m},d\mu_\nu').$ 
Let $p_{t, \nu}$ denote the integral kernel of the semigroup $(e^{-t \mathcal{\tilde{L}_\nu}})_{t>0}$, in the sense that 
\begin{equation*}
e^{-t\mathcal{\tilde{L}_\nu}} f(g) = \int_{\h_{n,m}}  p_{t, \nu}(g, g') \, f(g') \, d\mu_{\nu}'(g').   
\end{equation*}
Using Equation (3.6) of \cite{Hebisch-Mauceri-Meda-spectral-multipliers-drift-Lie-group-Math-Z-2005}, it is easy to verify that 
$$p_{t, \nu}(g, g') = e^{-t|\nu|^2} \, e^{-(b,a) \cdot (u,v)} \, e^{-(b,a) \cdot (u',v')} \, p_{t}(g'^{-1} g),
$$
where $g = (u,v,s), \, g' = (u',v',s') \in \h_{n,m}$. 

\medskip 
Now, for any multi-index $\gamma$ such that $|\gamma|=k \geq 1$, one can define the Riesz transform $\tilde{R}_{\gamma, \nu}$ via functional calculus as follows: 
$$ \tilde{R}_{\gamma, \nu} = \tilde{X}^\gamma (\mathcal{\tilde{L}}_{\nu})^{-k/2} = \frac{1}{\Gamma(k/2)} \int_{0}^{\infty} t^{\frac{k}{2}-1} \, \tilde{X}^\gamma e^{-t \, \mathcal{\tilde{L}_\nu}} \, dt. $$ 
The following result follows from \cite[Theorem 2]{Lohoue-Mustapha-drift-2004}.
\begin{theorem}[Lohou\'{e}--Mustapha \cite{Lohoue-Mustapha-drift-2004}] \label{thm:Lohoue-Mustapha-Riesz-Lp}
For any $1 < p < \infty$ and the multi-index $\gamma$ such that $|\gamma|=k \geq 1$, the Riesz transform $\tilde{R}_{\gamma,\nu} = \tilde{X}^\gamma (\mathcal{\tilde{L}}_{\nu})^{-k/2}$ is bounded on $L^p(\h_{n,m}, \, d \mu_\nu')$. 
\end{theorem}

In the particular case of $m=1$, the Heisenberg-Reiter group $\h_{n,1}$ is nothing but the Heisenberg group. Recently, Li--Sj\"ogren studied the endpoint case of $p=1$ on $\mathbb{H}_{n,1}$. They proved the following result. 
\begin{theorem}[Li--Sj\"ogren \cite{Heisenberg_group_drift}] \label{thm:LS-Riesz-Heisenberg}
On the Heisenberg group $\h_{n,1}$, the first order Riesz transform $\tilde{X}(\mathcal{\tilde{L}}_{\nu})^{-1/2}$ is of weak-type $(1,1)$ with respect to $d\mu_\nu'$, uniformly in $\nu$.
\end{theorem}

In \cite{Heisenberg_group_drift}, it was also proved that on $\h_{n,1}$, for any order three or more, not all the Riesz transforms are of weak-type $(1,1)$ and to the best of our knowledge, it is still an open problem whether the second-order Riesz transforms are of weak-type $(1,1)$ or not.


\subsection{Relation between the Grushin operator and the sub-Laplacian} 
\label{subsec:relation-Grushin-sub-Laplacian}
In this subsection, we recall that one can realise the gradient vector fields for the Grushin operator as images of the gradient vector fields of the Heisenberg-Reiter group $\h_{n,m}$ under a unitary representation. This kind of relationship is known to be quite useful in the analysis of several interesting operators, including the Grushin operators, as shown in \cite{Martini-Sikora-sharp-multiplier-Grushin-MRL-2012, Dziubanski-Jotsaroop-H1-BMO-Grushin-JFAA-2016, Riesz_transform_Robinson, Dziubanski-Sikora-Lie-group-approach-JLT-2021}. We recall here (without explanation) some of the useful facts, and for details we refer to \cite{Martini-Sikora-sharp-multiplier-Grushin-MRL-2012, Dziubanski-Jotsaroop-H1-BMO-Grushin-JFAA-2016}.

\medskip 
The group $\h_{n,m}$ acts on $\R^{n+m}$ by translations 
$$
\tau_{(u,v,s)}(x',x'') = (x'-v, \, x'' - u^Tx' - s + u^T v/2), 
$$
which induces a unitary representation $\sigma$ of $\h_{n,m}$ on $L^2(\R^{n+m})$ given by 
\begin{equation} \label{eq:sigma_of_f}
\sigma_{(u,v,s)} f(x',x'') = f\circ\tau_{(u,v,s)}^{-1}(x',x'') = f\left(x' + v, \, x'' + u^T x' + s + u^Tv / 2 \right). 
\end{equation}

It turns out that
\begin{equation} \label{eq:d_sigma}
d\sigma(\tilde{X}_{j,k}) = X_{j,k}, \quad d\sigma(\tilde{Y}_{j}) = X_{j}, \quad \text{and therefore} \quad
d\sigma(\mathcal{\tilde{L}}) = G.    
\end{equation}

The unitary representation $\sigma$ further induces a representation on the function space $L^1(\h_{n,m})$ in a natural way: for any $F \in L^1(\h_{n,m})$, one defines 
\begin{equation*}
\sigma(F)f(x) = \int_{\h_{n,m}} F(u,v,s) \, \sigma_{(u,v,s)} f(x) \, du \, dv \, ds = \int_{\R^{n+m}} \sigma(F)(x,y) \, f(y) \, dy, 
\end{equation*}
where
\begin{equation*}
\sigma(F)(x,y) = \int_{\R^{n \times m}} F\left(u, \, y'-x', \, y''-x''- u^T(x'+y') / 2 \right) \, du.
\end{equation*}
With this, the heat kernel of the Grushin operator can be obtained from the heat kernel of the Heisenberg-Reiter group by 
\begin{equation} \label{eq:relation_between_heat_kernel_of_Grushin_and_Reiter_group}
H_t(x,y) = \sigma(p_t)(x,y).    
\end{equation}

Before moving on, let us remark that with $\{\tilde{X}_{1,1}^R, \ldots,\tilde{X}_{n,m}^R, \, \tilde{Y_1}^R,\ldots,\tilde{Y_n}^R, \, \tilde{T_1}^R, \ldots,\tilde{T}_{m}^R\}$ denoting the canonical set of right-invariant gradient vector fields on $\h_{n, m}$, one also has 
\begin{align}
\label{eq:vector_fields_and_sigma}
d\sigma(\tilde{X}_{j,k})_x \, \sigma(F)(x,y) &= - \sigma(\tilde{X}_{j,k}^R \, F)(x,y), \\ 
\nonumber \text{and} \qquad d\sigma(\tilde{Y}_{j})_x \, \sigma(F)(x,y) &= -\sigma(\tilde{Y}_{j}^R \, F)(x,y).
\end{align} 


\subsection{Grushin operator with drift}
\label{subsec:Grushin-operator-with-drift}

Since $G = d\sigma(\mathcal{\tilde{L}})$, it is natural to define the Grushin operator with drift as the image of the sub-Laplacian with drift on $\h_{n, m}$ under the representation $d\sigma$. More precisely, given a non-zero vector $\nu = (a,b) \in \R^{n} \times \R^{nm}$, we define the Grushin operator with drift $G_\nu$ by 
\begin{equation} 
\label{def:Grushin-drift-via-projection}
G_\nu = d\sigma (\mathcal{\tilde{L}_\nu}),
\end{equation}
where  $\mathcal{\tilde{L}_\nu}$ is as in \eqref{def:sub-Laplacian_with_drift}.

\medskip 
Thanks to the identities of \eqref{eq:d_sigma}, the operator $G_\nu$ has the following natural expression: 
$$ G_\nu = G - 2 \, \nu \cdot X, $$
where the vector field $X$ is given by \eqref{def:grushin-vector-fields}. 

\medskip 
Interestingly, it turns out that for $G_\nu$ to be symmetric with respect to some positive measure on $\R^{n+m}$, we must have $b = 0$. This (and more) will be shown in the following proposition, which can be seen as an analogue of Proposition 3.1 of \cite{Hebisch-Mauceri-Meda-spectral-multipliers-drift-Lie-group-Math-Z-2005}. 

\begin{proposition} \label{prop:auto-def-grushin-drifted}
The Grushin operator with drift $G_\nu$ is symmetric with respect to a positive measure $\mu$ on $\R^{n+m}$ if and only if $b=0$ and upto a scalar multiple, the measure $\mu$ is given by $d\mu = d\mu_a = e^{2a\cdot x'} \, dx$. Consequently, the Grushin operator with drift becomes $G_a = G - 2 a \cdot \nabla_{x'}$ and it is essentially self-adjoint with respect to $d\mu_a = e^{2a\cdot x'} \, dx$.     
\end{proposition}
\begin{proof}
Let $\mu$ be a positive measure on $\R^{n+m}$ such that $G_\nu$ is symmetric with respect to $\mu$. For any $f,g \in \C_c^\infty(\R^{n+m})$, upon performing integration by parts, one gets 
\begin{align} \label{general-identity-symmetry} 
\int_{\R^{n+m}} (G_\nu f) \, \overline{g} \, d\mu = \int_{\R^{n+m}} f \, (\overline{G_\nu g}) \, d\mu + F(f,g,\mu). 
\end{align}
In the above expression we have 
$$ F(f,g,\mu)= \left( f, \, g\, (G-V) \mu  - 2 X' g \cdot X'\mu - 2 X'' g \cdot X''\mu + 4 (a\cdot X' g)\mu + 4(b\cdot X'' g) \mu \right), $$ 
where $V = - 2 \, \nu \cdot X$ and $(\cdot, \cdot)$ denotes the standard Euclidean inner product on $\R^{n+m}$. 

\medskip 
From \eqref{general-identity-symmetry}, it is clear that $G_\nu$ is symmetric on $L^2(\R^{n+m}, \, d\mu)$ if and only if $F(f,g,\mu)=0$ for all $f,g \in \C_c^\infty(\R^{n+m})$. But then, since the gradient vectors $X'$ and $X''$ are homogeneous of degree 1 (with respect to the non-isotropic dilations $\delta_r$), we must have 
\begin{align} \label{identity-solution-hypoelliptic-op} 
(G-V) \mu = 0. 
\end{align}

Note also that the operator $G-V$ is hypoelliptic (see Theorem 1.1 of \cite{Hypoelliptic_Hormander}). Hence, identity \eqref{identity-solution-hypoelliptic-op} implies that the measure $\mu$ has to be of the form $d\mu(x) = h(x) \, dx$ for some positive smooth function $h$ on $\R^{n+m}$. 

\medskip 
With this, we return to the identity \eqref{general-identity-symmetry} which becomes 
\begin{align}
\label{identity-symmetry-form-1} 
\int_{\R^{n+m}} (G_\nu f) \,  \overline{g} \, h \, dx =\int_{\R^{n+m}} f \, (\overline{G_\nu g}) \,h \,  dx,
\end{align}
for all $f,g \in \C_c^\infty(\R^{n+m})$. 

\medskip 
In particular, let us take $f(x)= \phi_1(x') \psi(x'')$ and $g(x)= \phi_2(x') \psi(x'')$ for some real-valued, compactly supported, smooth functions $\phi_1, \phi_2$ and $\psi$. Putting these in \eqref{identity-symmetry-form-1}, an easy calculation gives 
$$
\int_{\R^n} \left\{(\Delta_{x'} + 2a\cdot\nabla_{x'}) \, \phi_1(x') \right\} \phi_2(x') \, \Psi(x')\, dx' = \int_{\R^n} \left\{ (\Delta_{x'}+2a\cdot\nabla_{x'})\phi_2(x') \right\} \phi_1(x') \, \Psi(x') \, dx',
$$
where $\Psi(x')= \int_{\R^m} \psi(x'')^2 h(x', x'') \, dx'' $.

\medskip 
The identity above is nothing but the statement asserting that the operator $\Delta_{x'}+2a\cdot\nabla_{x'}$, that is, the Laplacian with drift, is symmetric on $\R^n$ with respect to the measure $\Psi(x') \, dx'$. We can now invoke Proposition 3.1 of \cite{Hebisch-Mauceri-Meda-spectral-multipliers-drift-Lie-group-Math-Z-2005} to conclude that the function $\Psi$ has to be 
$$ \Psi(x') = C_{\psi} \, e^{2a\cdot x'},$$ 
with $C_{\psi}$ a non-negative constant depending on $\psi$.

\medskip 
We thus have
$$
e^{-2a\cdot x'} \int_{\R^m} \psi(x'')^2 \, h(x',x'') \, dx'' = C_{\psi}.
$$

Taking partial derivatives in $x'$-variable, we get that 
$$
\int_{\R^m} \psi(x'')^2 \left\{ -2a_j \, e^{-2a\cdot x'} \, h(x',x'') + e^{-2a\cdot x'} \partial_{x'_j} \, h(x',x'')\right\} dx''= 0,
$$
for every $j=1,\dots, n$ and for all $\psi \in C_c^\infty(\R^m)$. 

\medskip 
Since the above identity holds true for all $\psi \in C_c^\infty(\R^m)$, we conclude that $$\partial_{x'_j} h(x',x'')= 2a_j h(x',x'')$$
for all $j=1,2,\dots,n$, which further implies that 
$$h(x',x'') = e^{2a\cdot x'}K(x'')$$ 
for some positive smooth function $K$ on $\R^m$. 

\medskip 
In view of the above form of $h$, identity \eqref{identity-symmetry-form-1} transforms to 
\begin{align} 
\label{identity-symmetry-form-2} 
\int_{\R^{n+m}} (G_\nu f) \, \overline{g} \, e^{2a\cdot x'} \, K(x'') \, dx = \int_{\R^{n+m}} f \, (\overline{G_\nu g}) \, e^{2a\cdot x'} \, K(x'') \,  dx, 
\end{align}
for all $f,g \in \C_c^\infty(\R^{n+m})$. 

\medskip 
Now, let us consider real-valued, compactly supported, smooth functions $f$ and $g$ of the form $f(x)= \phi_1(x') \, \psi_1(x'')$ and $g(x)= \phi_2(x') \, \psi_2(x'')$, and if we put them in \eqref{identity-symmetry-form-2}, then we would get 
\begin{align} 
\label{identity-symmetry-form-3} 
& \int_{\R^{n+m}} \phi_1(x') \, \phi_2(x') \, \psi_1(x'') \left\{ |x'|^2 \, \Delta_{x''} \psi_2(x'') + 2b\cdot X''\psi_2(x'') \right\} e^{2a\cdot x'} \, K(x'') \, dx' \, dx'' \\ 
\nonumber & \quad = \int_{\R^{n+m}} \phi_1(x') \, \phi_2(x') \, \psi_2(x'') \left\{ |x'|^2 \, \Delta_{x''} \psi_1(x'') + 2b\cdot X''\psi_1(x'') \right\} e^{2a\cdot x'} \, K(x'') \, dx' \, dx''. 
\end{align}
By defining the following constants (which depend on $a$, $\phi_1$ and $\phi_2$): 
\begin{align} 
\label{def:constants-intermediate-step}
C = \int_{\R^n} \phi_1(x') \, \phi_2(x') \, |x'|^2 e^{2a\cdot x'} \, dx' \quad \text{and} \quad C_{j,k} = 2 \int_{\R^n} \phi_1(x') \, \phi_2(x') \, x'_j \, e^{2a\cdot x'} \, dx',
\end{align} 
identity \eqref{identity-symmetry-form-3} can be rewritten as 
\begin{align} \label{eq:first_component_function_equal}
& C \int_{\R^m} \psi_1(x'') \, \Delta_{x''} \psi_2(x'') \, K(x'') \, dx'' + \sum_{j,k} b_{j,k} \, C_{j,k} \int_{\R^m} \psi_1(x'') \, \partial_{x''_k}\psi_2(x'') \, K(x'') \, dx'' \\ 
\nonumber & \quad = C \int_{\R^m} \psi_2(x'') \, \Delta_{x''} \psi_1(x'') \, K(x'') \, dx'' + \sum_{j,k} b_{j,k} \, C_{j,k} \int_{\R^m} \psi_2(x'') \, \partial_{x''_k} \psi_1(x'') \, K(x'') \, dx''. 
\end{align}

Now, by performing the integration by parts in the left-hand side of \eqref{eq:first_component_function_equal}, one gets 
\begin{align} \label{eq:first_component_function_equal-2}
\int_{\R^m} \psi_2(x'') \, \left\{ \psi_1(x'') \, \Theta_1(x'') + 2 \sum_k \partial_{x''_k}\psi_1(x'') \, \Theta_{2,k}(x'') \right\} dx'' = 0, 
\end{align}
where 
\begin{align} 
\label{eq:last_equation_which_gives_b=0}
\Theta_1(x'') &= C \, \Delta_{x''}K(x'') -\sum_{j,k} b_{j,k} \, C_{j,k} \, \partial_{x''_k} K(x'') \\ 
\nonumber \quad \text{and}
\quad \Theta_{2,k}(x'') &= C \, \partial_{x''_k} K(x'') - \sum_{j} b_{j,k} \, C_{j,k} \, K(x''). 
\end{align}

Since identity \eqref{eq:first_component_function_equal-2} is true for all $\psi_2 \in \C_c^\infty(\R^{m})$, it follows that 
$$\psi_1(x'') \, \Theta_1(x'') + 2 \sum_k \partial_{x''_k}\psi_1(x'') \, \Theta_{2,k}(x'') = 0.$$
But, this identity being true for all $\psi_1 \in \C_c^\infty(\R^{m})$, we can invoke the fact that each partial derivative $\partial_{x''_k}$ is homogeneous of degree 1, to conclude that $\Theta_1 = \Theta_{2,k} = 0$ for all $k$. 

\medskip 
Remember that the constants $C$ and $C_{j,k}$ appearing in \eqref{eq:last_equation_which_gives_b=0} are given in \eqref{def:constants-intermediate-step} and they are defined with the help of $a$ and arbitrary real-valued functions $\phi_1, \, \phi_2 \in \C_c^\infty(\R^n)$. At this stage, we specialise to some specific functions $\phi_1$ and $\phi_2$ as follows. Choose and fix a non-trivial, real-valued even function $A \in C_c^\infty(\R)$. Now, for a fixed $l \in \{1, 2, \dots, m\}$, consider 
\begin{align*} 
\phi_1(x') = \prod_{j=1}^n A(x'_j) \quad \text{and} \quad \phi_2(x') = x'_l \, e^{-2a\cdot x'} \prod_{j=1}^n A(x'_j). 
\end{align*}
Observe that for these specific functions, we have $C=0,$ as well as for all $k$, we have $C_{j,k} = 0$ for every $j \neq l$ and $C_{l,k} > 0$. But then, since $\Theta_{2,k} = 0$, we get 
$$b_{l,k} \, C_{l,k} \, K = 0.$$
But, the function $K$ is non-trivial. Therefore, we must have $b_{l,k} \, C_{l,k} = 0$, and hence $b_{l,k} = 0.$ Thus, we get $b=0$. 

\medskip 
Continuing, one can easily choose functions $\phi_1$ and $\phi_2$ for which $C \neq 0$. With any such constant $C$, if we put $b=0$ in \eqref{eq:first_component_function_equal}, then we would get 
$$
\int_{\R^m} \psi_1(x'') \, \Delta_{x''} \psi_2(x'') \, K(x'') \, dx'' = \int_{\R^m} \psi_2(x'') \, \Delta_{x''} \psi_1(x'') \, K(x'') \, dx'',
$$
for all $\psi_1, \psi_2 \in C_c^\infty(\R^m).$ 

\medskip 
Arguing in a manner similar to the ones above, one can show that the function $K$ has to be a (positive) constant. 

\medskip 
Summarising, we have shown that if the Grushin operator with drift $G_\nu$ is symmetric with respect to a positive measure $\mu$ on $\R^{n+m}$, then we must have $b= 0$, that is, 
$$
G_\nu = -\Delta_{x'} - |x'|^2 \Delta_{x''} -2a\cdot \nabla_{x'} = G-2a\cdot\nabla_{x'},
$$
which we will denote by $G_a$ and the measure $\mu$ has to be $\mu_a$ or its scalar (positive) multiple, where $d\mu_a = e^{2a \cdot x'} \, dx$. 

\medskip 
On the converse side, one has by direct calculations 
$$
\int_{\R^{n+m}} (G_a f) \, \overline{g}\, \, d\mu_a  = \int_{\R^{n+m}} \left\{ \nabla_{x'}f \cdot \overline{\nabla_{x'}g} + |x'|^2 \, \nabla_{x''}f \cdot \overline{\nabla_{x''}g} \right\} d\mu_a,
$$
for any $f, g \in C_c^\infty(\R^{n+m})$, which implies that $G_a$ with domain $C_c^\infty (\R^{n+m})$ is positive-definite and symmetric on $L^2(\R^{n+m},d\mu_a).$ 

\medskip 
Finally, for a given non-zero vector $a \in \R^n$, let us consider the operator $U_a$ defined on $L^2(\R^{n+m}, \, dx)$ by $U_a f(x',x'') = e^{-a \cdot x'} f(x', x'')$. Clearly, $U_a$ is an isometry from $L^2(\R^{n+m}, \, dx)$ onto $L^2(\R^{n+m}, \, d\mu_a)$. One can easily verify that
\begin{equation} \label{eq:Relation_between_Grushin_and_drift_Grushin}
G_a = U_a \, (G+|a|^2I) \, U_a^{-1}.
\end{equation}
Since $G$ is positive-definite and essentially self-adjoint on $L^2(\R^{n+m}, \, dx)$, it follows that $G_a$ is positive-definite and essentially self-adjoint on $L^2(\R^{n+m}, d\mu_a)$.
\end{proof}

Note that the operator $G_a$ generates a symmetric heat diffusion semigroup $(e^{-t \, G_a})_{t > 0}$, which is given by 
$$
e^{-t \, G_a} f(x) = \int_{\R^{n+m}} H_{t,a}(x,y) \, f(y) \, e^{2a \cdot y'} \, dy,
$$
where $H_{t,a}(x,y)=e^{-|a|^2 t} \, e^{-a \cdot (x'+y')} \, H_t(x,y)$ with $H_t(x,y)$ given in \eqref{eq:heat_kernel_in_hermite_form}. One can verify the above integral form and the claim that $(e^{-t \, G_a})_{t > 0}$ is a symmetric diffusion semigroup via the relations \eqref{eq:relation_between_heat_kernel_of_Grushin_and_Reiter_group}, \eqref{eq:Relation_between_Grushin_and_drift_Grushin} and the fact that $\mathcal{\tilde{L}_\nu}$ generates a convolution semigroup of probability measures (see Subsection \ref{sbsec:Heisenberg_Reiter_Group}). 

\medskip 
Given a multi-index $\alpha \in \mathbb{N}^n \times \mathbb{N}^{nm}$ such that $|\alpha|= k \geq 1$, via functional calculus we define the $k^{th}$ order Riesz transform $R_{\alpha,a}$ by 
\begin{equation} 
\label{def:Riesz-transforms-Grushin-drift} 
R_{\alpha,a} = X^\alpha G_a^{-k/2} = \frac{1}{\Gamma(k/2)} \int_{0}^{\infty} t^{\frac{k}{2}-1} \, X^\alpha e^{-t \, G_a} \, dt. 
\end{equation}


Recall that the asymptotic behavior of the Lebesgue measure of balls $B(x,r)$ corresponding to the Grushin metric $d$ is well known (see \eqref{eq:ball_volume}). We shall now compute an analogous asymptotic estimate of the measure of these balls with respect to the exponential measure at hand. 

\begin{lemma} 
\label{lem:ball-vol-est}
With $d\mu_a = e^{2a\cdot x'} \, dx$, we have 
\begin{align}
\label{main:ball-vol-est}
\mu_a(B(x,r)) \sim 
\left\{
\begin{array}{ll}
e^{2a \cdot x'} \, r^{n+m} \, (r+|x'|)^m, & \mbox{if } r \leq 1/|a| \\
|a|^{-(n+1)/2-m} \, e^{2(a \cdot x'+|a|r)} \, r^\frac{n-1}{2} \, (r+|x'|)^m, & \mbox{if } r>1/|a|.
\end{array} \right.  
\end{align} 
\end{lemma}
\begin{proof} 
We shall prove this lemma for $a = e_1$ only. One can deduce the result for other non-zero vectors $a$ with the help of rotation and dilation in a standard way. 

\medskip 
We shall make use of Grushin metric $d$ given by \eqref{Grushin_distance} and the ball volume estimate as stated in \eqref{eq:ball_volume}. Given $x=(x',x'') \in \R^{n+m}$ and $r>0$, let us consider the ball $B=B(x,r)$. When $r \leq 1$, the claimed estimate \eqref{main:ball-vol-est} follows easily from \eqref{eq:ball_volume}. For the same, note that if $r \leq 1$, then we have 
\begin{align*}
\mu_{e_1}(B) &= \int_{B} e^{2y'_1} dy' \, dy'' = e^{2x'_1} \int_{B} e^{2(y'_1-x'_1)} \, dy' \, dy''  \sim e^{2x'_1} \int_{B} dy' \, dy'' \sim e^{2x'_1} \, r^{n+m} \, (r+|x'|)^m.
\end{align*}
 
\medskip 
We shall now prove the estimate \eqref{main:ball-vol-est} in the case of $r>1$. For any $0<h<2r$, consider the set 
$$
E_h = \{(y',y'') \in B : y'_1 - x'_1 = r-h\}
$$
and let
$$
E_h^1 = E_h \cap A \qquad \text{and} \qquad E_h^2 = E_h \cap (B \setminus A),
$$
where $A=\{(y',y'') \in B : |x''-y''|^{\frac{1}{2}} \leq |x'|+|y'| \}$.
Then,
\begin{align}
\label{eq:measure_starting_point}
\mu_{e_1}(B) = e^{2x'_1} \int_{B} e^{2(y'_1-x'_1)} \, dy' \, dy'' &= e^{2x'_1} \int_{0}^{2r} e^{2(r-h)} |E_h| \, dh \\ 
\nonumber &=  e^{2x'_1} \int_{0}^{2r} e^{2(r-h)} (|E_h^1|+|E_h^2|) \, dh,    
\end{align} 
where $|\cdot|$ denotes the $(n+m-1)$-dimensional Lebesgue measure of a set. 

\medskip 
When $n \geq 2$, let us denote the points in $\R^n$ by $y' = (y'_1, y'_\perp)$, otherwise $y' = y'_1$ and one can make sense of the below arguments appropriately. One has 
$$
|x'-y'|-|x'_1-y'_1| = \frac{|x'-y'|^2-|x'_1-y'_1|^2}{|x'-y'|+|x'_1-y'_1|}\geq \frac{|x'_\perp-y'_\perp|^2}{2|x'-y'|}. 
$$

We shall prove the upper estimate and lower estimate in \eqref{main:ball-vol-est} one by one.

\medskip 
\textbf{\underline{Step I (Proof of $\lesssim$ in estimate \eqref{main:ball-vol-est})}:} 
Note that for any $y=(y',y'') \in E_h$, we have $d(x,y)-|x'_1-y'_1|< r-|r-h| \leq h$.

\medskip Now, if $y \in E_h^1$, then 
$
d(x,y)-|x'-y'|= \frac{|x''-y''|}{|x'|+|y'|}
$ 
implies that 
$$
h \geq d(x,y)-|x'_1-y'_1| \geq \frac{|x''-y''|}{|x'|+|y'|} + \frac{|x'_\perp-y'_\perp|^2}{2|x'-y'|}.
$$

In particular, we have  
$$
\frac{|x''-y''|}{|x'|+|y'|} \leq h \qquad \text{and} \qquad \frac{|x'_\perp-y'_\perp|^2}{2|x'-y'|} \leq h,
$$
which implies that 
$$
|x''-y''|\leq h(|x'|+|y'|) \leq h(|x'|+|y'-x'|+|x'|) < 2h(r+|x'|)
$$
and 
$$
|x'_\perp-y'_\perp|^2 \leq 2h|x'-y'| < 2hr.
$$

From the above, we conclude that 
\begin{align} \label{case-ii-est-1}
|E_h^1| \lesssim h^{\frac{n-1}{2}+m} \, r^{\frac{n-1}{2}} \, (r+|x'|)^m. 
\end{align}

\medskip On the other hand, if $y \in E_h^2$, then
$
d(x,y)-|x'-y'| = |x''-y''|^{1/2}
$
implies that 
$$
h \geq d(x,y)-|x'_1-y'_1| \geq |x''-y''|^{1/2} + \frac{|x'_\perp-y'_\perp|^2}{2|x'-y'|}. 
$$

In particular, we have  
$$
|x''-y''|^{\frac{1}{2}} \leq h \qquad \text{and} \qquad |x'_\perp-y'_\perp|^2 \leq 2h|x'-y'| < 2hr.
$$

From these estimates, we conclude that 
\begin{align} \label{case-ii-est-2} 
|E_h^2| \lesssim h^{\frac{n-1}{2}+2m} r^{\frac{n-1}{2}} \lesssim h^{\frac{n-1}{2}+m} r^{\frac{n-1}{2}} (r+|x'|)^m.
\end{align}

\medskip 
Making use of \eqref{case-ii-est-1} and \eqref{case-ii-est-2}, we get from \eqref{eq:measure_starting_point} that 
\begin{align*}
\mu_{e_1}(B) &= e^{2x'_1} \int_{0}^{2r} e^{2(r-h)} (|E_h^1|+|E_h^2|) \, dh \\ 
& \lesssim e^{2x'_1} \int_{0}^{2r} e^{2(r-h)} \, h^{\frac{n-1}{2}+m} \, r^{\frac{n-1}{2}} \, (r+|x'|)^m \, dh \\ 
& \leq r^{\frac{n-1}{2}} \, (r+|x'|)^m \, e^{2(x'_1+r)} \int_{0}^{\infty} e^{-2h} \, h^{\frac{n-1}{2}+m} \, dh \\ 
& \lesssim r^{\frac{n-1}{2}} \, (r+|x'|)^m \, e^{2(x'_1+r)}, 
\end{align*}
which is precisely the upper bound in the claimed estimate \eqref{main:ball-vol-est}. 

\medskip 
\textbf{\underline{Step II (Proof of $\gtrsim$ in estimate \eqref{main:ball-vol-est})}:} It suffices to show that there exists a constant $C>0$ such that 
\begin{align} \label{est:gtrsim-small-h}
|E_h^1| \geq C \, r^{\frac{n-1}{2}} (r+|x'|)^m,
\end{align}
for all $\frac{1}{4} \leq h \leq \frac{1}{2}$. 

To see this, note that if we have \eqref{est:gtrsim-small-h}, then from \eqref{eq:measure_starting_point}
we get 
\begin{align*}
\mu_{e_1}(B) & = e^{2x'_1} \int_{0}^{2r} e^{2(r-h)} (|E_h^1|+|E_h^2|) \, dh  \\
& \geq C \, e^{2(x'_1+r)} \, r^{\frac{n-1}{2}} \, (r+|x'|)^m \int_{\frac{1}{4}}^{\frac{1}{2}} e^{-2h} \, dh \sim e^{2(x'_1+r)} \, r^{\frac{n-1}{2}} \, (r+|x'|)^m,
\end{align*}
which is the claimed lower estimate. 

\medskip 
With $\frac{1}{4} \leq h \leq \frac{1}{2}$, we shall prove \eqref{est:gtrsim-small-h} in two separate cases: $|x'|>r$ and $|x'| \leq r$. 

\medskip 
Consider first the case when $|x'|>r$, and define the following set: 
$$\widetilde{E_{h,1}^1} := \left\{ y \in \R^{n+m} : 
|x'_\perp-y'_\perp|^2 < \frac{r}{16}, \quad |x''-y''| < \frac{|x'|}{16} \quad \text{and} \quad y'_1-x'_1=r-h \right\}.$$ 
We shall show that $\widetilde{E_{h,1}^1} \subseteq E_h^1$. For the same, note first that since $|x'| > r > 1$, we have $|x'|^{1/2} < |x'|$. Therefore, if $y \in \widetilde{E_{h,1}^1}$, then 
$|x''-y''|^{1/2} < \frac{|x'|^{1/2}}{4} < |x'|+|y'|$. Furthermore, 
\begin{align*}
d(x,y)-|x'_1-y'_1| &= \frac{|x''-y''|}{|x'|+|y'|} + |x'-y'| -|x'_1-y'_1| \\
& = \frac{|x''-y''|}{|x'|+|y'|} + \frac{|x'_\perp-y'_\perp|^2}{|x'-y'| \, +|x'_1-y'_1|} \\
& \leq \frac{|x'|/16}{|x'|} \, + \frac{r/16}{|x'_1-y'_1|} \\
& = \frac{1}{16} + \frac{r}{16(r-h)} < \frac{1}{16} + \frac{1}{8} < \frac{1}{4} \leq h,
\end{align*}
where we have used $r-h > r/2$ which follows simply from $h \leq 1/2 < r/2$. 

\medskip 
Thus, we get $d(x,y) < h+|x'_1-y'_1| = h+(r-h) = r,$ and therefore $y \in E_h^1$. Hence, $\widetilde{E_{h,1}^1} \subseteq E_h^1$, and this implies $|E_h^1| \geq (\frac{1}{16})^{\frac{n-1}{2}+m} \, r^{\frac{n-1}{2}} \, |x'|^m \sim r^{\frac{n-1}{2}} \, (r+|x'|)^m$.

\medskip 
On the other hand, in the case when $|x'| \leq r$, let us consider the following set: 
$$\widetilde{E_{h,2}^1} := \left\{ y \in \R^{n+m} : 
|x'_\perp-y'_\perp|^2 < \frac{r}{16}, \quad |x''-y''| < \frac{r}{16} \quad \text{and} \quad y'_1-x'_1=r-h \right\}.$$
We shall show that $\widetilde{E_{h,2}^1} \subseteq E_h^1$. For the same, note first that if $y \in \widetilde{E_{h,2}^1}$, then 
$$|x''-y''|^{1/2} < \frac{r^{1/2}}{4} < r-h = y'_1-x'_1 \leq |y' - x'| \leq |x'|+|y'|.$$
Thus, 
\begin{align*}
d(x,y) - |x'_1-y'_1| & = \frac{|x''-y''|}{|x'|+|y'|} + \frac{|x'_\perp-y'_\perp|^2}{|x'-y'| +|x'_1-y'_1|} \\
& \leq \frac{r/16}{r/2} \, + \frac{r/16}{|x'_1-y'_1|} \\
& = \frac{1}{8} + \frac{r}{16(r-h)} < \frac{1}{8} + \frac{1}{8} = \frac{1}{4} \leq h,
\end{align*}
where we have used $|x'|+|y'| \geq |x'-y'| \geq |x_1'-y_1'| = r-h > r/2$. 

\medskip 
From the above, we get once again that $d(x,y) < h+|x'_1-y'_1| = h+(r-h) = r,$ which implies that $y \in E_h^1$. Thus, $\widetilde{E_{h,2}^1} \subseteq E_h^1$, and therefore $|E_h^1| \geq \left(\frac{1}{16}\right)^{\frac{n-1}{2}+m} \, r^{\frac{n-1}{2}} \, r^m \sim r^{\frac{n-1}{2}} \, (r+|x'|)^m.$ This completes the proof of Lemma \ref{lem:ball-vol-est}. 
\end{proof}


\subsection{Rotation and Dilation}
\label{subsec:rotation-dilation}
Before embarking on the proofs of the main results, let us discuss the effect of rotations and dilations on the Riesz transforms $R_{\alpha,a}$ (defined in \eqref{def:Riesz-transforms-Grushin-drift}). With this it will often become sufficient, and technically convenient, to simply work with the drift vector $a = e_1$. 

\medskip 
Recall from equation \eqref{eq:heat_kernel_in_hermite_form} that 
$$
H_t((x',x''),(y',y'')) = (2\pi)^{-m} \int_{\R^m} k_{t,\lambda}(x',y') \, e^{-i\lambda\cdot (x''-y'')} \, d\lambda, 
$$
where $k_{t,\lambda}$ is given by \eqref{eq:heat_kernel_hermite}. 


\subsubsection{Rotation} 
Let $A=(a_{ij})$ be an $n \times n$ orthogonal matrix. 
It follows from $k_{t,\lambda}(Ax',Ay')=k_{t,\lambda}(x',y')$ that 
\begin{align*}
H_{t}((Ax',x''),(Ay',y'')) = H_{t}((x',x''),(y',y'')). 
\end{align*}

Now, if we define the action of $A$ on $\R^{n+m}$ by $Ax=(Ax',x'')$, then it is straightforward to verify that 
$$
H_{t,Aa}(Ax,Ay) = H_{t,a}(x,y),
$$
using which we have 
\begin{align*}
(G_{Aa})^{-k/2}(Ax,Ay) & =  \frac{1}{\Gamma(k/2)} \int_{0}^{\infty} t^{\frac{k}{2}-1} H_{t,Aa}(Ax,Ay) \, dt 
\\ & =\frac{1}{\Gamma(k/2)} \int_{0}^{\infty} t^{\frac{k}{2}-1} H_{t,a}(x,y) \, dt  
= (G_{a})^{-k/2}(x,y).
\end{align*}

Now, let us first consider the first-order Riesz transform $R_{\alpha, a}$, with $\alpha = e_j$ for some $j \in \{ 1, 2, \ldots, n \}$. We have 
\begin{align}
\label{rel:rotation-effect-1} 
\nonumber R_{e_j,a}(x,y) &= \partial_{x'_j} (G_{a})^{-1/2}(x,y) \\ 
\nonumber &= \partial_{x'_j} \left\{ (G_{Aa})^{-k/2}(Ax,Ay) \right\} \\ 
& = \sum_{l=1}^n a_{lj} (\partial_{x'_l}(G_{Aa})^{-1/2})(Ax,Ay) = \sum_{l=1}^n a_{lj} R_{e_l,Aa}(Ax,Ay).
\end{align}
On the other hand, if $\alpha = e_i$ for some $i \in \{ n+1 , \ldots, n + nm \}$, then for the corresponding pair $(j,k) \in \{ 1, \ldots, n \} \times \{ 1 , \ldots, m \}$, we have 
\begin{align}
\label{rel:rotation-effect-2} 
\nonumber R_{e_i,a}(x,y) &= R_{e_{(j,k)}, a}(x,y) = x'_j \partial_{x''_k} (G_{a})^{-1/2}(x,y) \\ 
\nonumber &= x'_j \partial_{x''_k}\left\{ (G_{Aa})^{-k/2}(Ax,Ay) \right\} \\ 
& = \sum_{l=1}^n a_{lj} (x'_l\partial_{x''_k}(G_{Aa})^{-1/2})(Ax,Ay) = \sum_{l=1}^n a_{lj} R_{e_{(l,k)}, Aa} (Ax,Ay).  
\end{align}

Relations \eqref{rel:rotation-effect-1} and \eqref{rel:rotation-effect-2} show how we can express first order Riesz transforms with drift vector $a$ as a finite linear combination of first order Riesz transforms with drift vector $Aa$. Recursively, one can write down relation between Riesz transforms of higher orders for drift vectors $a$ and $Aa$. 

\medskip 
Given a non-zero vector $a$, one can always choose an $n \times n$ orthogonal matrix $A$ such that $Aa = |a| e_1$. Note also from the relations \eqref{rel:rotation-effect-1} and \eqref{rel:rotation-effect-2} (and the same would hold true for their higher order analogues) that the coefficients in the sum on the right hand side are matrix entries that are uniformly bounded over orthogonal matrices. Therefore, in proving various boundedness aspects of the Riesz transforms, it will suffice to work with the scalar multiples of the unit vector $e_1$. 


\subsubsection{Dilation} 
Recall that the non-isotropic dilations in the context of the Grushin operator are given by $\delta_s (x) = \delta_s (x', \, x'') = (sx', \, s^2x'')$ for $s>0$. Now, 
\begin{align*}
H_{s^2 t}(\delta_sx, \, \delta_s y) & = (2\pi)^{-m}\int_{\R^m} k_{s^2 t,\lambda}(sx',sy') \, e^{-i\lambda\cdot (s^2x''- s^2y'')} \, d\lambda 
\\ & = (2\pi)^{-m} s^{-n} \int_{\R^m}k_{t, s^2\lambda}(x',y') \, e^{-i s^2 \lambda \cdot (x''-y'')} \, d\lambda 
\\ & = (2\pi)^{-m} s^{-n-2m}\int_{\R^m}k_{t, \lambda}(x',y') \, e^{-i \lambda \cdot (x''-y'')} \, d\lambda = s^{-n-2m} \, H_{t}(x,y),
\end{align*}
using which it can easily be verified that 
\begin{align*}
H_{t, \, se_1}(x,y) = s^{n+2m} H_{s^2t,e_1} (\delta_sx, \, \delta_s y),  
\end{align*}
and therefore  
\begin{align*}
(G_{se_1})^{-k/2}(x,y) & = \frac{1}{\Gamma(k/2)} \int_{0}^\infty t^{\frac{k}{2}-1} H_{t,se_1}(x,y) \, dt
\\ & = \frac{1}{\Gamma(k/2)} \int_{0}^\infty t^{\frac{k}{2}-1} s^{n+2m} H_{s^2t,e_1}(\delta_sx, \, \delta_s y)  \, dt
\\ & = \frac{s^{n+2m-k}}{\Gamma(k/2)} \int_{0}^\infty t^{\frac{k}{2}-1} H_{t,e_1}(\delta_sx, \, \delta_s y)  \, dt
\\ & = s^{n+2m-k} (G_{e_1})^{-k/2}(\delta_sx, \, \delta_s y).
\end{align*}

Hence, the kernel of the Riesz transform satisfies 
\begin{align*}
R_{\alpha,se_1}(x,y) & = X^\alpha (G_{se_1})^{-k/2}(x,y)
\\ & = s^{n+2m-k} X^\alpha \left\{ (G_{e_1})^{-k/2}(\delta_sx, \, \delta_s y) \right\}
\\ & = s^{n+2m} (X^\alpha (G_{e_1})^{-k/2})(\delta_sx, \, \delta_s y)
\\ &= s^{n+2m} R_{\alpha,e_1}(\delta_sx, \, \delta_s y),
\end{align*}
where we have made use of the fundamental fact that all the vector fields $X_j$ and $X_{j,k}$ are homogeneous of degree $1$ corresponding to the non-isotropic dilations. 

\medskip 
As a consequence, we have 
\begin{equation} \label{rel:dilation-effect}
R_{\alpha,se_1}f(x) = R_{\alpha,e_1} F(\delta_sx), 
\end{equation} 
where $f(x)=F(\delta_sx)$. 

\medskip 
The relation \eqref{rel:dilation-effect} implies that for any $1 \leq p < \infty$, we have 
\begin{align} \label{rel:dilation-effect-op-norm}
\|R_{\alpha,se_1}\|_{L^{p}(d\mu_{se_1}) \to L^{p}(d\mu_{se_1})} &= \|R_{\alpha,e_1}\|_{L^{p}(d\mu_{e_1}) \to L^{p}(d\mu_{e_1})} \\ 
\nonumber \text{and} \quad 
\|R_{\alpha,se_1}\|_{L^{p}(d\mu_{se_1}) \to L^{p,\infty}(d\mu_{se_1})} 
&= \|R_{\alpha,e_1}\|_{L^{p}(d\mu_{e_1}) \to L^{p,\infty}(d\mu_{e_1})}.
\end{align}

Given a drift vector $a$, we can just choose $s= |a|$, and then in view of the relations given by \eqref{rel:dilation-effect-op-norm}, it shall suffice to work with the drift vector $e_1$. 


\section{Regularized Riesz transforms and a transference principle} 
\label{sec:Reg-Riesz-transforms-transference-principle}

Motivated by the approach developed in \cite{Riesz_transform_Robinson}, in this section we define the regularized Riesz transforms associated with $G_{e_1}$ on $\R^{n+m}$ and $\mathcal{\tilde{L}}_{e_1}$ on $\h_{n,m}$. It can be shown that the weak and strong $L^p$-norms of the regularized Riesz transforms for $\mathcal{\tilde{L}}_{e_1}$ are uniformly bounded by that of the corresponding Riesz transforms.

\medskip 
The main reason to define these regularized Riesz transforms is that they can be expressed in the integral forms with integrable kernels (see Proposition \ref{prop:reg-Riesz-intergral-form} and Lemma \ref{prop:integrability_of_kernel}), on which we can apply a transference principle (see proposition \ref{prop:transference}) motivated by the one from \cite{Transference_methods_in_analysis}. Using these results, we prove Theorems \ref{thm:p_bdd} and \ref{thm:1_bdd-positive} in Subsection \ref{subsec:proofs-thms-positive}.


\subsection{Regularized Riesz transforms}
\label{subsec:Reg-Riesz-transforms}

Given $0 < \epsilon, \delta < 1$, let us consider the regularized Riesz transforms associated with $G_{e_1}$, which are defined as follows: 
$$
R_{\alpha,e_1,\epsilon,\delta} = X^\alpha
(\delta I +G_{e_1})^{-k/2} (I+\epsilon G_{e_1})^{-N},
$$
for $\alpha \in \mathbb{N}^{n} \times \mathbb{N}^{nm}$ such that $|\alpha| = k$. Here, $N$ is a large natural number which will be chosen later as per the requirement, and we shall not keep $N$ in the index of various notations. 

\medskip 
Similarly, on the Heisenberg-Reiter group, we define the regularized Riesz transforms associated with $\mathcal{\tilde{L}}_{e_1}$ by 
$$
\tilde{R}_{\tilde{\alpha},e_1,\epsilon,\delta} = \tilde{X}^{\tilde{\alpha}} (\delta I +\mathcal{\tilde{L}}_{e_1})^{-k/2} (I+\epsilon \mathcal{\tilde{L}}_{e_1})^{-N},
$$
where $\tilde{\alpha} = (\alpha_{n+1}, \ldots, \alpha_{nm}, \, \alpha_1, \ldots, \alpha_n)$ whenever $\alpha = (\alpha_1, \ldots, \alpha_n, \, \alpha_{n+1}, \ldots, \alpha_{nm})$. 

\medskip 
It turns out that the operator norm of each of the regularized Riesz transforms $\tilde{R}_{\tilde{\alpha},e_1,\epsilon,\delta}$ is bounded by the operator norm of the corresponding Riesz transform $\tilde{R}_{\tilde{\alpha},e_1}$. More precisely, 
\begin{proposition} 
\label{prop:Relation_between_two_riesz_gp_transforms}
For any $N \geq 1$ and $1 \leq p < \infty$, there exists a constant $C_p > 0$ (depending also on $k$, $N$, $m$ and $n$) such that 
\begin{align*}
\|\tilde{R}_{\tilde{\alpha},e_1, \epsilon,\delta}\|_{L^{p}(d\mu_{e_1}') \to L^{p}(d\mu_{e_1}')} & \leq C_p \, \|\tilde{R}_{\tilde{\alpha},e_1}\|_{L^{p}(d\mu_{e_1}') \to L^{p}(d\mu_{e_1}')}, \\ 
\text{and} \quad \|\tilde{R}_{\tilde{\alpha},e_1, \epsilon,\delta}\|_{L^{p}(d\mu_{e_1}') \to L^{p,\infty}(d\mu_{e_1}')} & \leq C_p \, \|\tilde{R}_{\tilde{\alpha},e_1}\|_{L^{p}(d\mu_{e_1}') \to L^{p,\infty}(d\mu_{e_1}')}.
\end{align*}
\end{proposition}

We omit the proof of the above proposition as it follows by repeating the arguments (almost verbatim) of page 465 of \cite{Riesz_transform_Robinson}. 

\medskip 
We shall now show that the regularized Riesz transforms are integral operators. For the same, let us consider 
\begin{equation}
\label{def:t-integral-constant}
B_{\epsilon, \delta, k}(t) = e^{-t} \int_{0}^t h^{N-1} \, e^{-h/\epsilon} \, (t-h)^{\frac{k}{2}-1} \, e^{-\delta (t-h)} \, dh, 
\end{equation}
and define the kernel 
\begin{equation}
\label{eq:k_epsilon_delta}
K_{\epsilon, \delta}(g) = \sum_{\substack{\tilde{\gamma} \leq \tilde{\alpha}\\ \gamma_j = \alpha_j \forall j \neq 1}} M_\gamma \, e^{-v_1} \int_{0}^\infty B_{\epsilon, \delta, k}(t) \, \tilde{X}^{\tilde{\gamma}} p_t(g) \, dt,
\end{equation} 
where $p_t$ denotes the heat kernel of the Heisenberg-Reiter group, and we have used the notations $g = (u,v,s) \in \h_{n,m}$ and the constant $\displaystyle M_\gamma = \binom{\alpha_1}{\gamma_1} \, \frac{(-1)^{\alpha_1 -\gamma_1} \,  \epsilon^{-N} }{\Gamma(N)\Gamma(k/2)}.$ 

\begin{proposition} 
\label{prop:reg-Riesz-intergral-form}
The regularized Riesz transforms $R_{\alpha,e_1,\epsilon,\delta}$ and $\tilde{R}_{\tilde{\alpha},e_1,\epsilon,\delta}$ satisfy 
\begin{equation} 
R_{\alpha,e_1,\epsilon,\delta} = \int_{\h_{n,m}} K_{\epsilon, \delta}(g) \, \sigma_{g^{-1}} \,  dg \quad \text{and} \quad \tilde{R}_{\tilde{\alpha},e_1,\epsilon,\delta} = \int_{\h_{n,m}} K_{\epsilon, \delta}(g) \, R(g^{-1}) \, dg
\end{equation}
where $R$ denotes the right translation defined as $R(g') f(g) = f(g g')$ and $\sigma$ is the unitary representation of $\h_{n,m}$ on $L^2(\R^{n+m})$ as in \eqref{eq:sigma_of_f}. 
\end{proposition}
\begin{proof}
Let us first consider the operator $ (\delta I +\mathcal{\tilde{L}}_{e_1})^{-k/2} (I+\epsilon \mathcal{\tilde{L}}_{e_1})^{-N}$. 
The kernel of this operator is given by (as in Section 2 of \cite{Lp_regularity}) 
\begin{equation} \label{identity:kernel-form-regularised-potential}
\frac{\epsilon^{-N}}{\Gamma(N)\Gamma(k/2)} \int_0^\infty \int_0^t p_{t,e_1}(g,g') \, h^{N-1} \, e^{-h/\epsilon} \, (t-h)^{\frac{k}{2}-1} \, e^{-\delta (t-h)} \, dh \, dt. 
\end{equation}

In view of \eqref{identity:kernel-form-regularised-potential}, we have 
\begin{align*}
&\tilde{R}_{\tilde{\alpha},e_1,\epsilon,\delta} f(g) = \int_{\h_{n,m}} \tilde{R}_{\tilde{\alpha},e_1,\epsilon,\delta}(g,g') \, f(g') \, e^{2v'_1} \, dg'
\\ & = \int_{\h_{n,m}} \left\{ \frac{\epsilon^{-N}}{\Gamma(N)\Gamma(k/2)} \int_0^\infty \int_0^t \tilde{X}^{\tilde{\alpha}} p_{t,e_1}(g,g') \, h^{N-1} \, e^{-h/\epsilon} (t-h)^{\frac{k}{2}-1} \, e^{-\delta (t-h)} \, dh \, dt \right\} f(g') \, e^{2v'_1} \, dg'
\\ & = \sum_{\substack{\tilde{\gamma} \leq \tilde{\alpha}\\ \gamma_j = \alpha_j \, \forall j \neq 1}} M_\gamma \int_{\h_{n,m}} e^{-v_1-v'_1} \left\{ \int_0^\infty B_{\epsilon, \delta, k}(t) \, (\tilde{X}^{\tilde{\gamma}} p_{t})(g'^{-1}g) \, dt \right\} f(g') \, e^{2v'_1} \, dg'
\\ & = \sum_{\substack{\tilde{\gamma} \leq \tilde{\alpha}\\ \gamma_j = \alpha_j \, \forall j \neq 1}} M_\gamma \int_{\h_{n,m}} e^{-v'_1} \left\{ \int_0^\infty B_{\epsilon, \delta, k}(t) \, (\tilde{X}^{\tilde{\gamma}} p_{t})(g') \, dt \right\} f(gg'^{-1}) \, dg'
\\ & = \int_{\h_{n,m}} K_{\epsilon, \delta}(g') \, R(g'^{-1})f(g) \, dg', 
\end{align*}
which completes the part of the claim for the operator $\tilde{R}_{\tilde{\alpha},e_1,\epsilon,\delta}$.  

\medskip 
On the other hand, let us consider 
\begin{align*}
& \int_{\h_{n,m}} K_{\epsilon, \delta}(g) \, \sigma_{g^{-1}} F(x',x'') \, dg \\ 
& = \sum_{\substack{\tilde{\gamma} \leq \tilde{\alpha} \\ \gamma_j = \alpha_j \, \forall j \neq 1}} M_\gamma \int_{\h_{n,m}} e^{-v_1} \left\{ \int_{0}^\infty B_{\epsilon, \delta, k}(t) \, (\tilde{X}^{\tilde{\gamma}} p_t)(g) \, dt \right\} \sigma_{g^{-1}} F(x',x'') \, dg 
\\ & = \sum_{\substack{\tilde{\gamma} \leq \tilde{\alpha} \\ \gamma_j = \alpha_j \, \forall j \neq 1}} M_\gamma \int_{\h_{n,m}} e^{-v_1} \left\{\int_{0}^\infty B_{\epsilon, \delta, k}(t) \, (\tilde{X}^{\tilde{\gamma}} p_t)(g) \, dt \right\} \\ 
& \qquad \qquad \qquad \qquad \qquad  F\left(x'-v, \, x''- s -u^Tx'+\frac{u^Tv}{2}\right)  \, du \, dv \, ds  \\ 
&=  \sum_{\substack{\tilde{\gamma} \leq \tilde{\alpha}\\ \gamma_j = \alpha_j \, \forall j \neq 1}} M_\gamma \int_{\h_{n,m}}  e^{v_1-x'_1} \left\{ \int_{0}^\infty B_{\epsilon, \delta, k}(t) \, (\tilde{X}^{\tilde{\gamma}} p_t) \left(u, x'-v, x''- s - \frac{u^Tx'}{2}-\frac{u^Tv}{2} \right) dt \right\} \\ 
& \qquad \qquad \qquad \qquad \qquad F(v,s) \, du \, dv \, ds \\ 
&= \sum_{\substack{\tilde{\gamma} \leq \tilde{\alpha}\\ \gamma_j = \alpha_j \, \forall j \neq 1}} M_\gamma \int_{\R^{n+m}} e^{-v_1-x'_1} \left\{ \int_{\R^{nm}} \int_{0}^\infty (\tilde{X}^{\tilde{\gamma}} p_t) \left(u,x'-v ,x''- s -\frac{u^Tx'}{2}-\frac{u^Tv}{2}\right) \right. \\ 
& \qquad \qquad \qquad \qquad \qquad \left. B_{\epsilon, \delta, k}(t) \, dt \, du \right\} F(v,s) \, e^{2v_1} \, dv \, ds, 
\end{align*}
and
\begin{align*}
& R_{\alpha,e_1,\epsilon,\delta} ((x',x''),(v,s)) \\
& = \frac{\epsilon^{-N}}{\Gamma(N) \Gamma (k/2)} \int_{0}^\infty \int_{0}^t  X^\alpha H_{t,e_1}((x',x''),(v,s)) \, h^{N-1} \, e^{-h/\epsilon} \, (t-h)^{\frac{k}{2}-1} \, e^{-\delta (t-h)} \, dh \, dt \\ 
& = \sum_{\substack{\gamma \leq \alpha\\ \gamma_j = \alpha_j \, \forall j \neq 1}} M_\gamma \, e^{-v_1-x'_1} \int_{0}^\infty B_{\epsilon, \delta, k}(t) \, X^\gamma H_{t}((x',x''),(v,s)) \, dt \\ 
& = \sum_{\substack{\tilde{\gamma} \leq \tilde{\alpha}\\ \gamma_j = \alpha_j \, \forall j \neq 1}} (-1)^{|\gamma|} M_\gamma \, e^{-v_1-x'_1} \int_{0}^\infty \left\{ \int_{\R^{nm}} ((\tilde{X}^{R})^{\tilde{\gamma}} p_t) \left(u, v-x', s -x''-\frac{u^Tx'}{2}-\frac{u^Tv}{2}\right) du \right\} \\ 
& \qquad \qquad \qquad \qquad \qquad \qquad \qquad B_{\epsilon, \delta, k}(t) \, dt \\ 
& = \sum_{\substack{\tilde{\gamma} \leq \tilde{\alpha}\\ \gamma_j = \alpha_j \, \forall j \neq 1}} M_\gamma \, e^{-v_1-x'_1} \int_{0}^\infty \left\{ \int_{\R^{nm}} (\tilde{X}^{\tilde{\gamma}} p_t) \left(u,x'-v,x''-s -\frac{u^Tx'}{2}-\frac{u^Tv}{2}\right) du \right\} B_{\epsilon, \delta, k}(t) \, dt,
\end{align*}
where we used \eqref{eq:vector_fields_and_sigma} in the third step. 

\medskip 
The above two identities together imply that 
$$
R_{\alpha,e_1,\epsilon,\delta} F(x) =\int_{\h_{n,m}} K_{\epsilon, \delta}(g) \, \sigma_{g^{-1}} F(x) \,  dg,
$$
which gives us the required integral forms. 
\end{proof}

Next, we shall prove that the kernel $K_{\epsilon, \delta}$ given in \eqref{eq:k_epsilon_delta} is integrable. This will be useful when we invoke the transference principle which we shall discuss in the next subsection. It was observed in \cite{Riesz_transform_Robinson} (see the discussion around inequality (4) in \cite{Riesz_transform_Robinson}) that their kernel $k_{\alpha; \nu, \epsilon} \in L^1(G)$. As mentioned there, the term $(I + \epsilon H)^{-n}$ helps in the integrability at the identity whereas the introduction of the $\nu$-term provides the integrability at infinity. In our context too, the proof of the fact that $K_{\epsilon, \delta} \in L^1(\h_{n,m}, \,e^{2v_1/p} \, dg)$, for any $1 \leq p < \infty$, is elementary and for the sake of convenience to readers we provide the details. 

\begin{lemma} \label{prop:integrability_of_kernel}
Let $K_{\epsilon, \delta}$ be as in \eqref{eq:k_epsilon_delta} with $N \geq \frac{Q}{2} + 1$ and $0 < \epsilon, \delta < 1$. Then, for any $1 \leq p < \infty$, we have $K_{\epsilon, \delta} \in L^1(\h_{n,m}, \,e^{2v_1/p} \, dg)$. 
\end{lemma}
\begin{proof}
Recall from \eqref{eq:k_epsilon_delta} that $K_{\epsilon, \delta}$ can be written as a finite sum with index $\tilde{\gamma}$ such that $\tilde{\gamma} \leq \tilde{\alpha}$ and $\gamma_j = \alpha_j$ for all $j \neq 1$. For technical convenience, in all the computations that follow, we shall always take one arbitrary term in $\tilde{\gamma}$, and from the proof it shall be clear that there is no loss in generality. Also, we shall make use of the bounds given by \eqref{est:gradient-group-heat-kernel-bounds} for a small $0 < \eta < 1/2$ with $(2+\eta)^2 = 4+\kappa$ such that $\delta - 4 \eta > 0$. We shall also make use of the notation $B(r,s)$ to denote the standard beta function $\int_0^1 h^{r-1} \, (1-h)^{s-1} \, dh$ which is obviously finite for all $r, s > 0$. 

\medskip 
We have 
\begin{align*} 
& \int_{\h_{n,m}} |K_{\epsilon, \delta}(g)| e^{2v_1/p} \, dg \\ 
\nonumber & \lesssim_{k,N,\epsilon} \int_{\h_{n,m}} \left| \int_{0}^\infty \int_{0}^t e^{-t} \, e^{-v_1} \, \tilde{X}^{\tilde{\gamma}} p_t(g) \, h^{N-1} \, e^{-h/\epsilon} \, (t-h)^{\frac{k}{2}-1} \, e^{-\delta (t-h)} \, dh \, dt\right| e^{2v_1/p} \, dg \\ 
& \lesssim_{\eta} \int_{\h_{n,m}} \int_{0}^\infty e^{-t} \, e^{\left( \frac{2}{p} - 1 \right)v_1} \, t^{-\frac{|\tilde{\gamma}|}{2}-\frac{Q}{2}} \, e^{-\frac{\varrho(g)^2}{(2+\eta)^2 t}} \, e^{-\delta t} \left( \int_{0}^t h^{N-1} \, (t-h)^{\frac{k}{2}-1} \, dh \right) dt \, dg \\ 
& \leq B(N,k/2) \int_{\h_{n,m}} \int_{0}^\infty e^{-t} \, e^{\varrho(g)} \, t^{\frac{k}{2}-\frac{|\tilde{\gamma}|}{2}+N-\frac{Q}{2}-1} \, e^{-\frac{\varrho(g)^2}{(2+\eta)^2 t}} \, e^{-\delta t} \, dt \, dg \\ 
& \lesssim_{k,N} \int_{\h_{n,m}} \int_{0}^\infty t^{\frac{k}{2}+N-\frac{|\tilde{\gamma}|}{2}-\frac{Q}{2}-1} \, e^{\frac{\eta \, \varrho(g)}{(2+\eta)}} \, e^{-\frac{1}{t} \left(\frac{\varrho(g)}{(2+\eta)} - t \right)^2} e^{-\delta t} \, dt \, dg \\ 
& = \int_{\h_{n,m}} \int_{0}^{\frac{\varrho(g)}{2(2+\eta)}} t^{\frac{k}{2}+N-\frac{|\tilde{\gamma}|}{2}-\frac{Q}{2}-1} \, e^{\frac{\eta \, \varrho(g)}{(2+\eta)}} \, e^{-\frac{1}{t} \left(\frac{\varrho(g)}{(2+\eta)} - t \right)^2} e^{-\delta t} \, dt \, dg \\
& \quad + \int_{\h_{n,m}} \int_{\frac{\varrho(g)}{2(2+\eta)}}^{\infty} t^{\frac{k}{2}+N-\frac{|\tilde{\gamma}|}{2}-\frac{Q}{2}-1} \, e^{\frac{\eta \, \varrho(g)}{(2+\eta)}} \, e^{-\frac{1}{t} \left(\frac{\varrho(g)}{(2+\eta)} - t \right)^2} e^{-\delta t} \, dt \, dg \\ 
& =: I_1+I_2.
\end{align*}

The analysis of both $I_1$ and $I_2$ is elementary as we show below. 
\begin{align*}
I_1 & = \int_{\h_{n,m}} \int_{0}^{\frac{\varrho(g)}{2(2+\eta)}} t^{\frac{k}{2}+N-\frac{|\tilde{\gamma}|}{2}-\frac{Q}{2}-1} \, e^{\frac{\eta \, \varrho(g)}{(2+\eta)}} \, e^{-\frac{1}{t} \left(\frac{\varrho(g)}{(2+\eta)} - t \right)^2} e^{-\delta t} \, dt \, dg \\ 
& \lesssim \int_{\h_{n,m}} \int_{0}^{\frac{\varrho(g)}{2(2+\eta)}} t^{\frac{k}{2}+N-\frac{|\tilde{\gamma}|}{2}-\frac{Q}{2}-1} \, e^{\frac{\eta \, \varrho(g)}{(2+\eta)}} \, e^{-\frac{\varrho(g)}{2(2+\eta)}} e^{-\delta t} \, dt \, dg \\ 
& = \int_{\h_{n,m}} \int_{0}^{\frac{\varrho(g)}{2(2+\eta)}} t^{\frac{k}{2}+N-\frac{|\tilde{\gamma}|}{2}-\frac{Q}{2}-1} \, e^{-\frac{(1-2 \eta)}{2(2+\eta)} \, \varrho(g)} \, e^{-\delta t} \, dt \, dg \\ 
& < \infty,
\end{align*}
and 
\begin{align*}
I_2 & = \int_{\h_{n,m}} \int_{\frac{\varrho(g)}{2(2+\eta)}}^{\infty} t^{\frac{k}{2}+N-\frac{|\tilde{\gamma}|}{2}-\frac{Q}{2}-1} \, e^{\frac{\eta \, \varrho(g)}{(2+\eta)}} \, e^{-\frac{1}{t} \left(\frac{\varrho(g)}{(2+\eta)} - t \right)^2} e^{-\delta t} \, dt \, dg \\ 
& \leq \int_{\h_{n,m}} \int_{\frac{\varrho(g)}{2(2+\eta)}}^{\infty} t^{\frac{k}{2}+N-\frac{|\tilde{\gamma}|}{2}-\frac{Q}{2}-1} \, e^{\frac{\eta \, \varrho(g)}{(2+\eta)}} \, e^{-\delta t} \, dt \, dg \\ 
& \leq \int_{\h_{n,m}} \int_{\frac{\varrho(g)}{2(2+\eta)}}^{\infty} t^{\frac{k}{2}+N-\frac{|\tilde{\gamma}|}{2}-\frac{Q}{2}-1} \, e^{- \frac{(\delta - 4\eta)}{4(2+\eta)} \varrho(g)} \, e^{-\frac{\delta}{2} t} \, dt \, dg \\ 
& < \infty. 
\end{align*}

Hence, $K_{\epsilon, \delta} \in L^1(\h_{n,m}, \, e^{2v_1/p} \, dg).$ 
\end{proof}

\subsection{A transference principle} 
\label{transference-principle}

In this subsection, we prove a transference result (Proposition \ref{prop:transference}), which will allow us to make use of the results developed in Subsection \ref{subsec:Reg-Riesz-transforms} so that we can get results for the Grushin operator with the help of those for the sub-Laplacian on the Heisenberg-Reiter group $\h_{n,m}$. 

\medskip 
Note also that the measure $d\mu_a (x) = e^{2a\cdot x'} \, dx$ on $\R^{n+m}$ corresponds to the measure $d\mu_{a}' = d\mu_{(a,0)}' = e^{2a\cdot v} \, dg$ on the Heisenberg-Reiter group $\h_{n,m}$. Let us mention that the proof of our result here, that is, Proposition \ref{prop:transference}, is mostly a reproduction of Theorems 2.4 and 2.6 of \cite{Transference_methods_in_analysis}, but since the mentioned results in \cite{Transference_methods_in_analysis} deal only with the Haar measure on certain amenable groups, we have decided to write the detailed proof in our context. 

\begin{proposition} \label{prop:transference}
Let $k \in L^1(\h_{n,m}, \, e^{2 \, a \cdot v/p} \, dg)$ for some $1 \leq p <\infty,$ and $\sigma$ be as in \eqref{eq:sigma_of_f}. For $f \in C_c^\infty(\R^{n+m}),$ let us define
$$ 
Tf(x) = \int_{\h_{n,m}} k(g) \, \sigma_{g^{-1}} f(x) \, dg. 
$$
Also, define for $\phi \in C_c^\infty(\h_{n,m})$,
$$
S\phi(g') = \int_{\h_{n,m}} k(g) \, R(g^{-1}) \, \phi(g') \, dg,
$$
where $R$ denotes the right translation defined as $R(g') f(g) = f(g g')$. Then, the strong (resp. weak) operator norm of $T$ is bounded by the strong (resp. weak) operator norm of $S$. More precisely, we have 
\begin{enumerate}[(i)]
\item $\displaystyle \|T\|_{L^p(\R^{n+m}, \, d\mu_a) \to L^p(\R^{n+m}, \, d\mu_a)} \leq \|S\|_{L^p(\h_{n,m}, \, d\mu'_a)\to L^p(\h_{n,m}, \, d\mu'_a)},$ 

\item $\displaystyle  \|T\|_{L^p(\R^{n+m}, \, d\mu_a) \to L^{p,\infty}(\R^{n+m}, \, d\mu_a)} \leq \|S\|_{L^p(\h_{n,m}, \, d\mu'_a)\to L^{p,\infty}(\h_{n,m}, \, d\mu'_a)},$
\end{enumerate}
where $d\mu_a(x)=e^{2 a \cdot x'} \, dx$ on $\R^{n+m}$ and $d\mu'_a (g)= e^{2 a \cdot v} \, dg$ on $\h_{n,m}$.
\end{proposition}
\begin{proof}
Note that with the given integrability condition on $k$, that is, $k \in L^1(\h_{n,m}, \, e^{2 \, a \cdot v/p} \, dg)$, both the operators $T$ and $S$ are well-defined, and in fact bounded on respective $L^p$-spaces, with their operator norm not exceeding $\|k\|_{L^1(\h_{n,m}, \, e^{2 \, a \cdot v/p} \, dg)}$. This is a direct consequence of Minkowski's integral inequality. 

\medskip 
We shall prove the proposition for $a = e_1$ only, and the proof for arbitrary $a$ follows on similar lines. We shall prove the claimed estimates only for compactly supported functions $k \in L^1(\h_{n,m}, \, e^{2v_1/p} \, dg)$. The result for arbitrary $k$ follows via standard density arguments. 

\medskip 
Let $K$ denote the support of function $k$. Given $\epsilon>0$, take $V$ to be a neighborhood of $0 \in \h_{n,m}$, having positive finite measure and such that 
$$
\frac{|VK^{-1}|}{|V|} \leq 1+\epsilon,
$$
where $|\cdot|$ denotes the Haar measure on $\h_{n,m}$ (see (2.1) in \cite{Transference_methods_in_analysis}). 

\medskip 
\begin{enumerate}[(i)]
\item Let us write $s_1 = \|S\|_{L^p(\h_{n,m}, \, d\mu'_{e_1}) \to L^p(\h_{n,m}, \, d\mu'_{e_1})}$. Now, consider
\begin{align*}
\int_{\R^{n+m}} |\sigma_{g'} Tf(x)|^p \, d\mu_{e_1}(x) &= \int_{\R^{n+m}} \left|Tf \left(x'+v',x'' + s' + u'^Tx' +\frac{u'^Tv'}{2}\right)\right|^p d\mu_{e_1}(x)  
\\ & = \int_{\R^{n+m}} |Tf(x',x'')|^p \, e^{-2v'_1} \, d\mu_{e_1}(x), 
\end{align*}
which implies that 
\begin{align*}
\int_{\R^{n+m}} |Tf(x',x'')|^p \, d\mu_{e_1}(x) =  e^{2v'_1} \int_{\R^{n+m}} |\sigma_{g'} Tf(x)|^p \, d\mu_{e_1}(x).
\end{align*}

Integrating both sides of the above identity over $g' \in V$, we get 
\begin{align*}
\|Tf\|_{L^p(\R^{n+m}, \, d\mu_{e_1})}^p &= \frac{1}{|V|} \int_{V} e^{2v'_1} \int_{\R^{n+m}} |\sigma_{g'} Tf(x)|^p \, d\mu_{e_1}(x) \, dg' 
\\ &=  \frac{1}{|V|} \int_{V} e^{2v'_1} \int_{\R^{n+m}} \left|\int_{\h_{n,m}} k(g) \, \sigma_{g'g^{-1}} f(x) \, dg \right|^p d\mu_{e_1}(x) \, dg'
\\ &=  \frac{1}{|V|}\int_{\R^{n+m}} \int_{V} \left|\int_{\h_{n,m}} k(g) \, \sigma_{g'g^{-1}} f(x) \, dg \right|^p e^{2v'_1} \, dg' \, d\mu_{e_1}(x) 
\\ & =  \frac{1}{|V|}\int_{\R^{n+m}} \int_{\h_{n,m}} \left|\int_{\h_{n,m}} k(g) \, \sigma_{g'g^{-1}} f(x) \, \chi_{VK^{-1}}(g'g^{-1}) \, dg \right|^p e^{2v'_1} \, dg' \, d\mu_{e_1}(x) 
\\ & \leq \frac{s_1^p}{|V|}\int_{\R^{n+m}} \int_{\h_{n,m}} \left| \sigma_{g}f(x) \, \chi_{VK^{-1}}(g) \right|^p e^{2v_1} \, dg \, d\mu_{e_1}(x) 
\\ & = \frac{s_1^p}{|V|}\int_{\h_{n,m}} |\chi_{VK^{-1}}(g)| \int_{\R^{n+m}} |\sigma_{g}f(x)|^p \, d\mu_{e_1}(x) \,  e^{2v_1} \, dg 
\\ & = \frac{s_1^p}{|V|}\int_{\h_{n,m}} |\chi_{VK^{-1}}(g)| \int_{\R^{n+m}} |f(x)|^p   \, d\mu_{e_1}(x) \, dg 
\\ & = \frac{s_1^p}{|V|} \, |VK^{-1}| \, \|f\|_{L^p(\R^{n+m}, \, d\mu_{e_1})}^p 
\\ & \leq s_1^p \, \|f\|_{L^p(\R^{n+m}, \, d\mu_{e_1})}^p \, (1+\epsilon).
\end{align*}

Since $\epsilon$ is arbitrary, we get 
$$
\|Tf\|_{L^p(\R^{n+m}, \, d\mu_{e_1})} \leq   s_1 \, \|f\|_{L^p(\R^{n+m}, \, d\mu_{e_1})},
$$ 
which proves part (i).

\medskip 
\item Let $s_2 = \|S\|_{L^p(\h_{n,m}, \, d\mu'_{e_1}) \to L^{p,\infty}(\h_{n,m}, \, d\mu'_{e_1})}$. Now, define the following sets: 
\begin{align*}
A(\lambda) &= \left\{ x \in \R^{n+m}: |Tf(x)| > \lambda \right\}, \\ 
\text{and} \quad A_{g'}(\lambda) &= \left\{ x \in \R^{n+m}: \left|\int_{\h_{n,m}} k(g) \, \sigma_{g' g^{-1}} f(x) \, dg \right| > \lambda \right\}.
\end{align*}
Clearly, $A(\lambda) = A_{e}(\lambda),$ where $e$ is identity of $\h_{n,m}.$

\medskip 
Also, let us write
$$
F(\lambda) = \left\{ (g',x) \in V \times \R^{n+m}: \left|\int_{\h_{n,m}} k(g) \sigma_{g' g^{-1}} f(x) \, dg \right| > \lambda \right\} . 
$$

Observe that $\mu_{e_1}(A(\lambda)) = e^{2v_1^\prime} \mu_{e_1}(A_{g'}(\lambda)).$ Therefore, upon integrating both sides over $g' \in V,$ we would get 
\begin{align*}
\mu_{e_1}(A(\lambda)) &= \frac{1}{|V|} \int_{V} e^{2v_1^\prime} \, \mu_{e_1}(A_{g'}(\lambda)) \, dg' \\
&= \frac{1}{|V|} \int_{\h_{n,m}} e^{2v_1^\prime}\left( \int_{\R^{n+m}} \chi_{F(\lambda)}(g',x) \, d\mu_{e_1}(x) \right) dg' \\
&= \frac{1}{|V|} \int_{\R^{n+m}} \left( \int_{\h_{n,m}} e^{2v_1^\prime} \, \chi_{F(\lambda)}(g',x) \, dg' \right) d\mu_{e_1}(x) \\
&= \frac{1}{|V|} \int_{\R^{n+m}} \mu'_{e_1} \left(\left\{g' \in V: \left|\int_{\h_{n,m}} k(g) \, \sigma_{g' g^{-1}} f(x) \, dg \right| > \lambda \right\} \right) d\mu_{e_1}(x) \\
& \leq \left(\frac{s_2}{\lambda} \right)^p \frac{1}{|V|} \int_{\R^{n+m}} \left(\int_{VK^{-1}} |\sigma_{g} f(x)|^p \, e^{2v_1} \, dg \right) d\mu_{e_1}(x) \\
& = \left(\frac{s_2}{\lambda} \right)^p \frac{1}{|V|} \int_{VK^{-1}} \left(\int_{\R^{n+m}} |\sigma_{g} f(x)|^p \, d\mu_{e_1}(x) \right) e^{2v_1} \, dg \\
& = \left(\frac{s_2}{\lambda} \right)^p \frac{1}{|V|} \int_{VK^{-1}} \left(\int_{\R^{n+m}} |f(x)|^p \, d\mu_{e_1}(x) \right) dg \\
& = \left(\frac{s_2}{\lambda} \right)^p \frac{|VK^{-1}|}{|V|} \, \|f\|^p_{L^p(\R^{n+m}, \, d\mu_{e_1})} \\
& \leq \left(\frac{s_2}{\lambda} \right)^p \|f\|^p_{L^p(\R^{n+m}, \, d\mu_{e_1})} \, (1+\epsilon). 
\end{align*}

Since, this is true for every $\epsilon>0,$ we have 
$$
\sup_{\lambda > 0} \, \lambda \, \mu_{e_1}(A(\lambda))^{1/p} \leq s_2 \, \|f\|_{L^p(\R^{n+m}, \, d\mu_{e_1})},
$$
which proves part (ii).
\end{enumerate}
\end{proof}


\subsection{Proofs of Theorems \ref{thm:p_bdd} and \ref{thm:1_bdd-positive}} \label{subsec:proofs-thms-positive} 

We are now ready to prove Theorems \ref{thm:p_bdd} and \ref{thm:1_bdd-positive} with the help of the transference principle discussed in Subsection \ref{transference-principle} applied on the regularized Riesz transforms from Subsection 
\ref{subsec:Reg-Riesz-transforms}.

\begin{proof}[Proof of Theorem \ref{thm:p_bdd}]
Let us take and fix $N \geq \frac{Q}{2} + 1$. With that, we have 
\begin{align} 
\label{est:riesz-transforms-comparative-bounds}
\|R_{\alpha,e_1, \epsilon, \delta} \|_{L^p(d\mu_{e_1}) \to L^p(d\mu_{e_1})} 
& \leq \|\tilde{R}_{\tilde{\alpha}, e_1, \epsilon, \delta}\|_{L^p(d\mu_{e_1}') \to L^p(d\mu_{e_1}')} \\ 
\nonumber & \leq C_p \| \tilde{R}_{\tilde{\alpha},e_1} \|_{L^p(d\mu_{e_1}') \to L^p(d\mu_{e_1}')} \leq C_{\alpha, p} < \infty, 
\end{align}
where the first inequality follows from Proposition \ref{prop:transference} (in view of Lemma \ref{prop:integrability_of_kernel} and Proposition \ref{prop:reg-Riesz-intergral-form}), the second inequality follows from Proposition \ref{prop:Relation_between_two_riesz_gp_transforms}, and the third inequality is nothing but the fact that the Riesz transforms on the Heisenberg-Reiter groups are $L^p$-bounded for $1<p<\infty$ (as stated in Theorem \ref{thm:Lohoue-Mustapha-Riesz-Lp}). 

\medskip 
Now, note that for any $f \in \bigcap_{\alpha} D(X^{\alpha})$, we can write 
$$ X^\alpha f = R_{\alpha,e_1,\epsilon,\delta} (I+\epsilon G_{e_1})^{N} (\delta I+ G_{e_1} )^{\frac{k}{2}} f, $$
and therefore we can make use of \eqref{est:riesz-transforms-comparative-bounds} to conclude that 
\begin{align*}
\|X^\alpha f\|_{L^p(d\mu_{e_1})} & \leq C_{\alpha, p} \, \|(I+\epsilon G_{e_1})^{N} (\delta I+ G_{e_1} )^{\frac{k}{2}} f\|_{L^p(d\mu_{e_1})}. 
\end{align*}

Taking limit as $\epsilon \to 0$ in the right hand side of the above inequality, we get 
\begin{equation} 
\label{est:delta-pending} 
\|X^\alpha f\|_{L^p(d\mu_{e_1})} \leq C_{\alpha, p} \, \| (\delta I+ G_{e_1} )^{k/2} f \|_{L^p(d\mu_{e_1})}. 
\end{equation}

\medskip 
Since the estimate in \eqref{est:delta-pending} is true for all small $\delta > 0$, it is natural to expect that with $\delta \to 0$ we should have 
\begin{equation*} 
\|X^\alpha f\|_{L^p(d\mu_{e_1})} \leq C_{\alpha, p} \, \| G_{e_1}^{k/2} f \|_{L^p(d\mu_{e_1})}, 
\end{equation*}
which indeed is true as we explain here. 

\medskip 
It suffices to show that 
$$
\lim_{\delta \to 0}\|(\delta I + G_{e_1})^{\gamma} f -G_{e_1}^\gamma f\|_{L^{p}(d\mu_{e_1})} = 0
$$
for all $\gamma > 0$. 

\medskip 
Obviously, the above identity is true whenever $\gamma$ is a positive integer, so one is left to verify it only for $\gamma \in (0,1)$. In fact, one can show that 
$$
\lim_{\delta \to 0}\|(\delta I + G_{e_1})^{\gamma} - G_{e_1}^\gamma \|_{L^{p}(d\mu_{e_1}) \to L^{p}(d\mu_{e_1})} = 0
$$
for every $\gamma \in (0,1)$. We do so following the ideas from the proofs of Lemma 4.2 of \cite{Heat_kernel_and_Riesz_transform_on_nilpotent_Lie_groups} and Section II.3.2 of \cite{Elliptic_operators_and_Lie_groups}as follows. 

\medskip 
Let us fix a $\gamma \in (0,1)$, and write the constant $C_\gamma^{-1} = \int_{0}^{\infty} s^{-\gamma} (1+s )^{-1} \, ds$. Then, by functional calculus, we have 
\begin{align*}
& (\delta I + G_{e_1})^{\gamma}  -G_{e_1}^\gamma  \\
& = C_\gamma \int_{0}^{\infty} s^{-\gamma} \left\{(\delta I+ G_{e_1})(I+ s(\delta I + {G_{e_1}}))^{-1}  - G_{e_1}(I + s G_{e_1})^{-1}  \right\} ds \\
& = C_\gamma \int_{0}^{\infty} s^{-\gamma-1} \left\{(\delta I+ G_{e_1})(s^{-1} I+ (\delta I + {G_{e_1}}))^{-1}  - G_{e_1}(s^{-1} I + G_{e_1})^{-1}  \right\} \, ds \\
& = C_\gamma \int_{0}^{\infty} s^{-\gamma-2} \int_{0}^{\infty} e^{-\frac{t}{s}} \left\{(1-e^{-t(\delta I + G_{e_1})})  - (1-e^{-tG_{e_1}}) \right\} \, dt \, ds \\
& = C_\gamma \int_{0}^{\infty} s^{-\gamma-2} \int_{0}^{\infty} e^{-\frac{t}{s}} \, (1 - e^{-t\delta}) \, e^{-tG_{e_1}}  \, dt \, ds \\
& = C_\gamma \int_{0}^{\infty} s^{-\gamma-1} \int_{0}^{\infty} e^{-r} (1-e^{-rs \delta}) e^{-rs
G_{e_1}}  \, dr \, ds.
\end{align*}

\medskip 
Since $\displaystyle \sup_{s>0} \|e^{-s \, G_{e_1}} \|_{L^{p}(d\mu_{e_1}) \to L^{p}(d\mu_{e_1})} \leq 1,$ we get 
\begin{align*}
&\|(\delta I + G_{e_1})^{\gamma} -G_{e_1}^\gamma \|_{L^{p}(d\mu_{e_1}) \to L^{p}(d\mu_{e_1})} \\
& \leq C_\gamma \int_{0}^{\infty} s^{-\gamma-1} \int_{0}^{\infty} e^{-r} (1-e^{-rs \delta}) \, \|e^{-rs
G_{e_1}} \|_{L^{p}(d\mu_{e_1}) \to L^{p}(d\mu_{e_1})} \, dr \, ds \\
& \leq C_\gamma \int_{0}^{\infty} s^{-\gamma-1} \int_{0}^{\infty} e^{-r} (1-e^{-rs \delta}) \, dr \, ds \\
& = C_\gamma \int_{0}^{\infty} s^{-\gamma-1} \left(1-\frac{1}{1+s\delta}\right) \, ds \\
& = C_\gamma \, \delta \int_{0}^{\infty} s^{-\gamma} (1+s \delta)^{-1} \, ds \\
& =  C_\gamma \, \delta^{\gamma} \int_{0}^{\infty} s^{-\gamma} (1+s )^{-1} \, ds \\
& =  \delta^{\gamma}, 
\end{align*}
which implies that 
$$
\lim_{\delta \to 0}\|(\delta I + G_{e_1})^{\gamma} -G_{e_1}^\gamma \|_{L^{p}(d\mu_{e_1}) \to L^{p}(d\mu_{e_1})} = 0.
$$

\medskip 
Thus we have shown that 
\begin{equation*} 
\|X^\alpha f\|_{L^p(d\mu_{e_1})} \leq C_{\alpha, p} \, \| G_{e_1}^{k/2} f \|_{L^p(d\mu_{e_1})}, 
\end{equation*}
and upon replacing $f$ by $G_{e_1}^{-k/2} f$, the claim of the theorem follows. 
\end{proof}

Our proof of Theorem \ref{thm:1_bdd-positive} follows from arguments similar to those in the previous proof. 

\begin{proof}[Proof of Theorem \ref{thm:1_bdd-positive}] 
We have to show that when $m=1$, the first order Riesz transforms are weak-type $(1,1)$. This can once again be proved using the transference from the Heisenberg group $\h_{n,1}$. In doing so, let us note that for $p=1$ we have estimate analogous to  
\eqref{est:riesz-transforms-comparative-bounds} given by 
\begin{align*} 
\|R_{\alpha,e_1, \epsilon, \delta} \|_{L^1(d\mu_{e_1}) \to L^{1,\infty}(d\mu_{e_1})} 
& \leq \|\tilde{R}_{\tilde{\alpha}, e_1, \epsilon, \delta}\|_{L^1(d\mu_{e_1}') \to L^{1,\infty}(d\mu_{e_1}')} \\ 
& \lesssim \| \tilde{R}_{\tilde{\alpha},e_1} \|_{L^1(d\mu_{e_1}') \to L^{1,\infty}(d\mu_{e_1}')} < \infty, 
\end{align*}
and here the last inequality holds true from Theorem \ref{thm:LS-Riesz-Heisenberg}. 

\medskip 
With the above estimate, one can essentially repeat the rest of the proof of Theorem \ref{thm:p_bdd} to conclude that the first order Riesz transforms $R_{\alpha, e_1}$, with $|\alpha|=1$, are of weak-type $(1,1)$ with respect to $d\mu_{e_1}$. 
\end{proof}


\section{Proof of Theorem \ref{thm:1_bdd-negative}}
\label{sec:1_bdd-negative}

In this section, we shall deduce Theorem \ref{thm:1_bdd-negative} from Theorem 1.1 of \cite{Li-Sjogren-drift-sharp-endpoint-Euclidean-Canad-2021} via a suitable transference. For the same, we shall show that for any $f \in C_c^\infty(\R^{n+m})$, appropriately scaled Riesz transform $\partial^k_{x_1'} \, (G_{e_1/R})^{-k/2} f$ converges pointwise to the Riesz transform $\partial^k_{x'_1} \, (\Delta_{e_1})^{-k/2} \, f$. Here, $\Delta_{e_1}$ stands for the standard Laplacian with drift $e_1$ on $\R^{n+m}$. 

\medskip 
Let us fix $(\xi',\xi'') \in \R^{n+m}$ such that $|\xi'| = |\xi''| =1,$ and consider the Euclidean translation of suitable functions on $\R^{n+m}$ by $(\xi',\xi'')$ given by 
$$ Uf(x',x'') = f(x'+\xi',x''+\xi'').$$ 
Clearly, $U^{-1} f(x',x'')= f(x'-\xi', x''-\xi'').$ 

\medskip  
Also, for any $R>0,$ let $\Lambda_R$ stand for the standard (isotropic) dilation on $\R^{n+m}$, that is, 
$$ \Lambda_R \, f(x', x'') = f(Rx', Rx''). $$ 
Clearly, $\Lambda_R^{-1} \, f(x', x'')= f \left( R^{-1} x', R^{-1} x'' \right).$ 

\begin{proposition} 
\label{prop:pointwise-convergence}
For any $k \geq 1$ and $f \in C_c^\infty(\R^{n+m})$, the following pointwise convergence holds true: 
\begin{equation} \label{eq:Riesz_transform_pointwise_convergence}
\lim_{R \to 0^+} \left( \Lambda_{R} \, U \right) \left( \partial^k_{x_1'} \, (G_{e_1/R})^{-k/2} \right) \left( \Lambda_{R} \, U \right)^{-1} f(x', x'') = \partial^k_{x'_1} \, (\Delta_{e_1})^{-k/2} \, f(x',x''), 
\end{equation}
\end{proposition}
where $\Delta_{e_1}$ denotes the Laplacian with drift on the Euclidean space $\R^{n+m}$. 

\medskip 
Before proving Proposition 
\ref{prop:pointwise-convergence}, let us see how we can use it to prove Theorem \ref{thm:1_bdd-negative}.

\begin{proof}[Proof of Theorem \ref{thm:1_bdd-negative}] 
As mentioned earlier, in view of rotation and dilation arguments, it suffices to work in the case of the drift vector $e_1$. We shall show that for any $k \geq 3$, the Riesz transform $R_{ke_1, e_1} = \partial^k_{x_1'}(G_{e_1})^{-k/2}$ is not of weak-type $(1,1)$ with respect to $d\mu_{e_1}$. In doing so, we shall make use of the result of Li--Sj\"ogren \cite[Theorem 1.1]{Li-Sjogren-drift-sharp-endpoint-Euclidean-Canad-2021} on the Euclidean space $\R^{n+m}$ which guarantees that $\partial^k_{x_1'}(\Delta_{e_1})^{-k/2}$ is of weak-type $(1,1)$ with respect to $d\mu_{e_1}$ if and only if $k \leq 2.$ 

\medskip 
Assume to the contrary that for some $k \geq 3$, the Riesz transform $R_{ke_1, e_1} = \partial^k_{x_1'}(G_{e_1})^{-k/2}$ is of weak-type $(1,1)$ with respect to $d\mu_{e_1}$. Recall also from the identity \eqref{rel:dilation-effect} that 
$$ 
R_{ke_1, \, e_1/R} = \delta^{-1}_R \, R_{ke_1, \, e_1} \, \delta_R. 
$$

With $\xi = (\xi', \xi'') \in \R^{n+m}$ and operators $U$ and $\Lambda_R$ as in the beginning of this section, for any $f \in C_c^\infty(\R^{n+m})$, we have 
\begin{align*}
\mu_{e_1}&\left\{x \in \R^{n+m}: \left|\left( (\Lambda_{R} \, U) \, R_{ke_1, \, e_1/R} \, (U^{-1} \Lambda_{R}^{-1}) f \right) (x)\right|> s\right\} 
\\ & = 
\mu_{e_1}\left\{x \in \R^{n+m}: \left|\left( \Lambda_R \, U \, \delta^{-1}_R \right) \, R_{ke_1, \, e_1} \, \left( \delta_R \, U^{-1} \, \Lambda^{-1}_R \right) f (x)\right| > s\right\}
\\ & = e^{-\frac{2\xi'_{1}}{R}} \, R^m \, \mu_{e_1}\left\{ x \in \R^{n+m}: \left|R_{ke_1, \, e_1}  \left( \delta_R \, U^{-1} \, \Lambda^{-1}_R f \right) (x)\right| > s\right\}
\\ & \lesssim e^{-\frac{2\xi'_{1}}{R}}  \frac{R^m}{s} \int_{\R^{n+m}} \left| \delta_R \, U^{-1} \, \Lambda^{-1}_R f(x) \right| e^{2x'_1} \, dx' \, dx''
\\ & = \frac{1}{s} \|f\|_{L^1(d\mu_{e_1})}. 
\end{align*}

But then, invoking Proposition \ref{prop:pointwise-convergence} and Fatou's lemma, one would get 
\begin{align*} 
& \mu_{e_1} \{x \in \R^{n+m}: | \partial^k_{x_1'}(\Delta_{e_1})^{-k/2} f(x)|> s\} 
\\ & = \mu_{e_1}\left\{x \in \R^{n+m}: \left|\lim_{R \to 0}\left( (\Lambda_{R} \, U) \, R_{ke_1, \, e_1/R} \, (U^{-1} \Lambda_{R}^{-1}) f \right) (x) \right|> s\right\} 
\\ & \leq \liminf_{R \to 0} \mu_{e_1} \left\{x \in \R^{n+m}: \left|\left( (\Lambda_{R} \, U) \, R_{ke_1, \, e_1/R} \, (U^{-1} \Lambda_{R}^{-1}) f \right) (x)\right|> s\right\} 
\\ & \lesssim \frac{1}{s} \|f\|_{L^1(d\mu_{e_1})},
\end{align*}
which contradicts Theorem 1.1 of \cite{Li-Sjogren-drift-sharp-endpoint-Euclidean-Canad-2021}, and this completes the proof of Theorem \ref{thm:1_bdd-negative}. 
\end{proof}

We are only left with the claim made in Proposition 
\ref{prop:pointwise-convergence}, which we prove below. 

\begin{proof}[Proof of Proposition \ref{prop:pointwise-convergence}] 
For convenience, we shall make use of the notations $t_h$ and $c_h$ to denote the hyperbolic tangent and the hyperbolic cotangent respectively. More precisely, $t_h (s) := \tanh (s)$ and $c_h (s) := \coth (s)$. 

\medskip 
Let us fix $f \in C_c^\infty(\R^{n+m})$ and recall that 
\begin{align} \label{eq:taking_derivative_inside_here}
\nonumber \partial_{x_1'}^k \, (G_{e_1/R})^{-k/2} \, f(x) 
& = \frac{1}{\Gamma(k/2)} \int_{0}^{\infty} t^{\frac{k}{2}-1} \, \partial_{x_1'}^k \, e^{-t \, G_{e_1/R}} f (x) \, dt \\ 
& = \frac{R^{k}}{\Gamma(k/2)} \int_{0}^{\infty} t^{\frac{k}{2}-1} \, \partial_{x_1'}^k \, e^{-t R^2 \, G_{e_1/R}} f (x) \, dt, 
\end{align}
so we shall analyse the action of $e^{-t R^2 \, G_{e_1/R}}$ on $f$. 

\medskip 
Let us write $\phi_{R, \, \xi} (y', y'') = e^{y'_1} \, \Lambda_R U f (y', y'') = e^{y'_1} \, f(Ry'+\xi',Ry''+\xi'')$. Now, 
\begin{align*}
& e^{-t R^2 \, G_{e_1/R}} \, f(x) \\ 
& = \int_{\R^{n+m}} H_{t R^2, e_1/R}(x,y) \, f(y) \, e^{\frac{2y_1'}{R}} \, dy \\
& = \int_{\R^{n+m}} e^{\frac{2\xi_1'}{R}} \,  H_{t R^2, e_1/R}(x, y + \xi) \, f(y+\xi) \, e^{\frac{2y_1'}{R}} \, dy \\
& = R^{n+m} \int_{\R^{n+m}}  e^{\frac{2\xi_1'}{R}} \, H_{t R^2, e_1/R}(x,(Ry'+\xi',Ry''+\xi'')) \, f(Ry'+\xi',Ry''+\xi'') \, e^{2y_1'} \, dy \\
& =R^{n+m} \int_{\R^{n+m}}  e^{\frac{\xi_1'}{R}} \, e^{-t} \, e^{\frac{-x_1'}{R}} \, H_{t R^2}(x,(Ry'+\xi',Ry''+\xi'')) \, \phi_{R, \, \xi} (y', y'') \, dy \\
& = (2\pi R)^{-m} \int_{\R^{n+m}} \int_{\R^m} e^{\frac{\xi_1'}{R}} \, e^{-t} \, e^{\frac{-x_1'}{R}} \, e^{i\lambda'' \cdot \left(\frac{y''}{R}-\frac{x''}{R^2}+\frac{\xi''}{R^2}\right)} \left(\frac{|\lambda''|}{2\pi \sinh{(2t|\lambda''|})}\right)^{n/2} \phi_{R, \, \xi} (y', y'') \\
& \quad \exp{\left(-\frac{|\lambda''| \, c_h{(2t|\lambda''|)}}{2} \left|\frac{x'-\xi'}{R}-y'\right|^2 \right)} \, \exp{\left(-|\lambda''| \, t_h{(t|\lambda''|)} \, \frac{x'}{R}\cdot \left(y'+\frac{\xi'}{R}\right)\right)} d\lambda'' \, dy, 
\end{align*}
with notations and formulas following the discussion surrounding \eqref{eq:heat_kernel_in_hermite_form}. 

\medskip 
In the final expression, by applying the Euclidean Fourier transform in $y''$-variable on $\R^m$, and the Euclidean Parseval's identity in $y'$-variable on $\R^n$, one can deduce that 
\begin{align*}
& e^{-t R^2 \, G_{e_1/R}} \, f(x) \\ 
&= (2\pi)^{-m} \int_{\R^{n+m}}  e^{\frac{\xi_1'}{R}} \, e^{-t} \, e^{\frac{-x_1'}{R}} \, e^{-i\lambda'' \cdot \left(\frac{x''-\xi''}{R}\right)} \, (\cosh{(2tR|\lambda''|)})^{-n/2} \, e^{2\pi i \frac{x'-\xi'}{R}\cdot\lambda'} \\
& \qquad \qquad \widehat{\phi_{R, \xi}} \left(\lambda', \frac{-\lambda''}{2\pi }\right) \exp{ \left(- \frac{|\lambda''| \, t_h{(tR|\lambda''|)}}{R} \left( 1 - \frac{t_h{(tR|\lambda''|)}}{2 \, c_h{(2tR|\lambda''|)}} \right) |x'|^2 \right)} \\ 
& \qquad \qquad \exp{\left(\frac{-2\pi^2|\lambda'|^2}{R \, |\lambda''| \, c_h{(2tR|\lambda''|)}}\right)} \, \exp{ \left( -2\pi i \, \frac{t_h{(tR|\lambda''|)}}{R \,  c_h{(2tR|\lambda''|)}} \, x'\cdot\lambda' \right)} d\lambda. 
\end{align*}

\medskip Now, when we put the above expression in the identity \eqref{eq:taking_derivative_inside_here}, the action of the derivative $\partial_{x_1'}^k$ involves the Leibniz principle and thus we obtain that $\partial_{x_1'}^k \, ( G_{e_1/R})^{-k/2} \, f(x)$ can be expressed as a finite linear combination of terms of the following type:  
\begin{align} \label{eq:taking_derivative_inside_here-2}
& \frac{(2\pi)^{-m} \, R^k}{\Gamma(k/2)}  \int_{0}^\infty \int_{\R^{n+m}} t^{\frac{k}{2}-1} \, e^{\frac{\xi_1'}{R}} \, e^{-t} \, e^{\frac{-x_1'}{R}} \, e^{-i\lambda'' \cdot \left(\frac{x''-\xi''}{R}\right)} \, e^{2\pi i \frac{x'-\xi'}{R}\cdot\lambda'} \,(\cosh{(2tR|\lambda''|)})^{-n/2} \\
& \nonumber \quad \widehat{\phi_{R, \xi}} \left(\lambda', \frac{-\lambda''}{2\pi }\right) \exp{\left(- \frac{|\lambda''| \, t_h{(tR|\lambda''|)}}{R} \left( 1 - \frac{t_h{(tR|\lambda''|)}}{2 \, c_h{(2tR|\lambda''|)}} \right) |x'|^2 \right)} \exp{\left(\frac{-2\pi^2|\lambda'|^2}{R|\lambda''| \, c_h{(2tR|\lambda''|)}}\right)} \\ 
& \nonumber \quad  \, \exp{\left(-2\pi i\frac{ t_h{(tR|\lambda''|)}}{R \, c_h{(2tR|\lambda''|)}} x'\cdot\lambda' \right)} \left(-\frac{2|\lambda''|\, t_h{(tR|\lambda''|)}}{R}\left(1-\frac{t_h{(tR|\lambda''|)}}{2 \, c_h{(2tR|\lambda''|)}}\right)\right)^{k_2} \\ 
& \nonumber \left(-\frac{1-2\pi i \lambda'_1}{R}-\frac{2|\lambda''| \, t_h{(tR|\lambda''|)} }{R}\left(1-\frac{t_h{(tR|\lambda''|)}}{2 \, c_h{(2tR|\lambda''|)}}\right) x_1' - 2\pi i \frac{t_h{(tR|\lambda''|)} \lambda'_1}{R \, c_h{(2tR|\lambda''|)}}\right)^{k_1}
\, d\lambda \, dt,  
\end{align}
where $k_1 + k_2 \leq k$ and $2 k_2 \leq k$. Also, in the mentioned finite linear combination, the constant coefficient against the leading term (for $k_1 = k$ and $k_2 = 0$) is exactly equal to 1. This is particularly important, as we shall see later when we take the limit as $R \to 0$. 

\medskip 
In order to see that taking the derivative under the integral sign is in fact valid, just note that since $t_h{(s)} \leq 1$ for all $s>0$, it is straightforward to see that all the terms appearing in the sum in \eqref{eq:taking_derivative_inside_here-2} are dominated by 
\begin{align*}
e^{\frac{(\xi_1'-x_1')}{R}} \, R^{k-(k_1+k_2)} \int_{0}^\infty \int_{\R^{n+m}} t^{\frac{k}{2}-1} \, e^{-t} & \, \left( 1 + |\lambda'|^{k_1} + |\lambda''|^{k_1} \, |x'_1|^{k_1} \right) |\lambda''|^{k_2} \left|\widehat{\phi_{R, \xi}} \left(\lambda', \frac{-\lambda''}{2\pi }\right) \right|
d\lambda \, dt, 
\end{align*}
and this integral converges absolutely, thanks to $f \in C_c^\infty(\R^{n+m})$. 

\medskip 
Thus, in view of \eqref{eq:taking_derivative_inside_here-2}, we have that $\left( \Lambda_{R} \, U \right) \left( \partial^k_{x_1'} \, (G_{e_1/R})^{-k/2} \right) \left( \Lambda_{R} \, U \right)^{-1} f(x', x'')$ can be expressed as a finite linear combination of terms of the following type: 
\begin{align} \label{eq:before_taking_R_to_zero}
& \frac{(2\pi)^{-m}}{\Gamma(k/2)} \, R^{k - k_1 - k_2} \int_{0}^\infty \int_{\R^{n+m}}  t^{\frac{k}{2}-1}   e^{-t} e^{-x_1'} e^{-i\lambda'' \cdot x''} \, e^{2\pi i x'\cdot\lambda'} \, (\cosh{(2tR|\lambda''|)})^{-n/2}  \\
& \nonumber \qquad \widehat{F} \left(\lambda', \frac{-\lambda''}{2\pi }\right) \exp{ \left(- R |\lambda''| \, t_h{(tR|\lambda''|)} \left( 1 - \frac{t_h{(tR|\lambda''|)}}{2 \, c_h{(2tR|\lambda''|)}} \right) \left|x'+\frac{\xi'}{R}\right|^2 \right)} \\ 
& \nonumber \qquad  \exp{\left(\frac{-2\pi^2|\lambda'|^2}{R|\lambda''| \, c_h{(2tR|\lambda''|)}}\right)} \, \exp{\left(-2\pi i  \frac{t_h{(tR|\lambda''|)}}{R \, c_h{(2tR|\lambda''|)}} (Rx'+\xi')\cdot\lambda' \right)} \\ 
& \nonumber \qquad \left(-1+2\pi i \lambda'_1 -2|\lambda''| \, t_h{(tR|\lambda''|)}\left(1 - \frac{t_h{(tR|\lambda''|)}}{2 \, c_h{(2tR|\lambda''|)}}\right)(Rx_1'+\xi_1') - 2\pi i \frac{ t_h{(tR|\lambda''|)}\lambda'_1}{c_h{(2tR|\lambda''|)}}  \right)^{k_1} \\
& \nonumber \qquad \left(-2|\lambda''| \, t_h{(tR|\lambda''|)} \left(1 - \frac{t_h{(tR|\lambda''|)}}{2 \, c_h{(2tR|\lambda''|)} }\right)\right)^{k_2}  d\lambda \, dt, 
\end{align}
where $F(x',x'') = f(x',x'') \, e^{x_1'}$, and as mentioned earlier, in the finite linear combination, the constant coefficient of the leading term (for $k_1 = k$ and $k_2 = 0$) is exactly equal to 1. 

\medskip 
We now take the limit as $R \to 0$ in \eqref{eq:before_taking_R_to_zero} and pass it inside the integrals which is permitted and one can verify the same using the standard generalised dominated convergence theorem. Note also that as $R \to 0$, all the integrands in terms corresponding to $k_1 < k$ in \eqref{eq:before_taking_R_to_zero} will tend to 0, and one would get 
\begin{align} 
\label{pointwise-convergence-final-argument} 
\nonumber & \lim_{R\to 0} \left( \Lambda_{R} \, U \right) \left( \partial^k_{x_1'} \, (G_{e_1/R})^{-k/2} \right) \left( \Lambda_{R} \, U \right)^{-1} f(x', x'') \\
\nonumber & = \frac{(2\pi)^{-m}}{\Gamma(k/2)} \int_{0}^\infty \int_{\R^{n+m}} t^{\frac{k}{2}-1} \, e^{-t} \, e^{-x_1'} \, e^{-i\lambda'' \cdot x''} \, e^{2\pi i x'\cdot\lambda'} \, e^{-|\lambda''|^2t} \, e^{-4t\pi^2|\lambda'|^2} \, \left(-1+2\pi i \lambda'_1\right)^{k} \\
\nonumber & \qquad \qquad \qquad \qquad \widehat{F} \left(\lambda', -\lambda''/(2\pi) \right) d\lambda \, dt \\
& = \frac{1}{\Gamma(k/2)} \int_{0}^\infty \int_{\R^{n+m}} t^{\frac{k}{2}-1} \, e^{-t} \, e^{-x_1'} \, e^{2\pi i x \cdot \lambda} \, e^{-4t \pi^2|\lambda|^2} \, \widehat{F} (\lambda) \left(-1+2\pi i \lambda'_1\right)^{k} \, d\lambda \, dt. 
\end{align} 

\medskip 
The final expression in \eqref{pointwise-convergence-final-argument} is nothing but $\partial_{x_1'}^k (\Delta_{e_1})^{-k/2}f(x',x'')$ as we now show. For the same, note that with $F(y',y'') = f(y',y'') \, e^{y_1'}$, we have 
\begin{align*}
e^{-t\Delta_{e_1}} f(x) &= (4\pi t)^{-(n+m)/2} \int_{\R^{n+m}} e^{-t} \, e^{-x'_1-y'_1} \, \exp{\left(-|x-y|^2/(4t) \right)} \, f(y) \, e^{2y'_1} \, dy \\
& = \int_{\R^{n+m}} e^{-t} \, e^{-x'_1} e^{2\pi i x \cdot \lambda} \, e^{-4\pi^2t|\lambda|^2} \widehat{F}(\lambda) \, d\lambda,
\end{align*}
which implies that 
\begin{align*}
\partial_{x'_1}^k (\Delta_{e_1})^{-k/2}f(x) 
&= \frac{1}{\Gamma(k/2)} \int_{0}^\infty \int_{\R^{n+m}} t^{\frac{k}{2}-1} \, e^{-t} e^{-x'_1} \, e^{2\pi i x \cdot \lambda}   
e^{-4\pi^2 t|\lambda|^2}
\widehat{F}(\lambda) \left( -1 + 2 \pi i \lambda'_1 \right)^{k}
\, d\lambda \, dt, 
\end{align*}
and hence we have from \eqref{pointwise-convergence-final-argument} that 
\begin{align*}
\lim_{R\to 0} \left( \Lambda_{R} \, U \right) \left( \partial^k_{x_1'} \, (G_{e_1/R})^{-k/2} \right) \left( \Lambda_{R} \, U \right)^{-1} f(x', x'') = \partial_{x_1'}^k (\Delta_{e_1})^{-k/2}f(x',x''). 
\end{align*}
This completes the proof of Proposition \ref{prop:pointwise-convergence}. 
\end{proof}


\section*{Acknowledgments}
First author is grateful to the Indian Institute of Science Education and Research Bhopal for the Senior Research Fellowship. 

\providecommand{\bysame}{\leavevmode\hbox to3em{\hrulefill}\thinspace}
\providecommand{\MR}{\relax\ifhmode\unskip\space\fi MR }
\providecommand{\MRhref}[2]{%
  \href{http://www.ams.org/mathscinet-getitem?mr=#1}{#2}
}
\providecommand{\href}[2]{#2}

\end{document}